\documentclass{article}
\usepackage{tikz}
\usetikzlibrary{patterns}
\usetikzlibrary{arrows.meta}
\tikzset{
 mydarrows/.style={scale=0.3,
 baseline=-2pt,
 >={Latex [open,length=4pt]},
 <->
 }
}

\tikzset{
 myarrows/.style={scale=0.3,
 baseline=-2pt,
 >={Latex [open,length=4pt]},
 ->
 }
}

\newcommand\biarrow[1][]{\mathbin{\tikz[mydarrows,#1]{\draw[double](0,0)--++(2,0);}}}
\newcommand\strongarrow[1][]{\mathbin{\tikz[myarrows,#1]{\draw[double](0,0)--++(1.5,0);}}}

\usepackage{tikz}

\newcommand{\Prop}{\mathtt{Prop}}
\newcommand{\Sf}{\mathtt{Sf}}

\newcommand{\LIT}{\mathsf{Lit}}
\newcommand{\Var}{\mathsf{Var}}
\newcommand{\QBel}{\mathsf{QBel}}

\renewcommand{\P}{\mathcal{P}}

\newcommand{\BD}{\mathsf{BD}}

\newcommand{\LBD}{\mathscr{L}_\mathsf{BD}}

\newcommand{\PrG}{\mathtt{P}}

\newcommand{\Be}{\mathtt{B}}
\newcommand{\EG}{\mathtt{E}_\mathsf{G}}

\newcommand{\weakrightarrow}{\rightarrowtriangle}
\newcommand{\weakleftrightarrow}{\leftrightarrowtriangle}
\newcommand{\weakcoimplies}{\multimap}
\newcommand{\coimplies}{\Yleft}

\newcommand{\Gsquare}{\mathsf{G}^2}
\newcommand{\Gsquareorder}{\mathsf{G}^2(\rightarrow,\coimplies)}
\newcommand{\GsquareNelson}{\mathsf{G}^2(\weakrightarrow,\weakcoimplies)}
\newcommand{\LGsquareNelson}{\mathscr{L}_{\mathsf{G}^2(\weakrightarrow,\weakcoimplies)}}
\newcommand{\LGsquareorder}{\mathscr{L}_{\mathsf{G}^2(\rightarrow,\coimplies)}}
\newcommand{\LbiG}{\mathscr{L}_\biG}

\newcommand{\HGsquareorder}{\mathcal{H}\mathsf{G}^2(\rightarrow,\coimplies)}
\newcommand{\HGsquareNelson}{\mathcal{H}\mathsf{G}^2(\weakrightarrow,\weakcoimplies)}
\newcommand{\MCB}{\mathsf{MCB}}
\newcommand{\NMCB}{\mathsf{NMCB}}
\newcommand{\CPL}{\mathsf{CPL}}
\newcommand{\LCPL}{\mathscr{L}_\CPL}
\newcommand{\LQP}{\mathscr{L}_\QP}

\newcommand{\LMCB}{\mathscr{L}_\MCB}
\newcommand{\LNMCB}{\mathscr{L}_\NMCB}
\newcommand{\HHB}{\mathcal{H}\mathsf{HB}}
\newcommand{\HB}{\mathsf{HB}}
\newcommand{\Rfde}{\mathbf{R}_\mathsf{fde}}
\newcommand{\HMCB}{\mathcal{H}\mathsf{MCB}}
\newcommand{\HNMCB}{\mathcal{H}\mathsf{NMCB}}

\newcommand{\toporder}{\top_{\mathbf{1}}}
\newcommand{\botorder}{\bot_{\mathbf{0}}}
\newcommand{\topNelson}{\top_{\mathsf{N}}}
\newcommand{\botNelson}{\bot_{\mathsf{N}}}
\newcommand{\Crtn}{\mathtt{C}}

\newcommand{\QPG}{\mathsf{QPG}}
\newcommand{\QG}{\mathsf{QG}}
\newcommand{\LQG}{\mathscr{L}_\QG}

\newcommand{\HQPG}{\mathcal{H}\QPG}
\newcommand{\HQG}{\mathcal{H}\QG}
\newcommand{\HQP}{\mathcal{H}\QP}

\newcommand{\QP}{\mathsf{QP}}
\newcommand{\SIF}{\mathsf{SIF}}

\newcommand{\biG}{\mathsf{biG}}
\newcommand{\Gtriangle}{\mathsf{G}\triangle}
\newcommand{\HbiG}{\mathcal{H}\mathsf{biG}}

\newcommand{\Grightarrow}{\rightarrow_{\mathsf{G}}}

\newcommand{\Gcoimplies}{\coimplies_{\mathsf{G}}}
\newcommand{\triangleorder}{\triangle^{\mathbf{1}}}
\newcommand{\triangleNelson}{\triangle^{\mathsf{N}}}

\usepackage{mathrsfs} 
\usepackage{graphicx}
\usepackage{comment}
\usepackage{amsmath}
\usepackage{amsthm}
\usepackage{amssymb}
\usepackage[colorlinks=true,linkcolor=blue,citecolor=green,urlcolor=orange]{hyperref}
\usepackage[margin=2.5cm]{geometry}
\usepackage{stmaryrd}
\usepackage{bussproofs}
\usepackage{mathtools}
\usepackage[all]{xy}
\usepackage{todonotes}
\usepackage{forest}
\usepackage{enumitem}
\usepackage{longtable}
\usepackage{verbatim}
\usepackage{authblk}
\usepackage{soul}
\newtheorem{theorem}{Theorem}[section]
\newtheorem{proposition}[theorem]{Proposition}
\newtheorem{corollary}[theorem]{Corollary}
\newtheorem{lemma}[theorem]{Lemma}
\theoremstyle{definition}
\newtheorem{definition}[theorem]{Definition}
\theoremstyle{remark}
\newtheorem{remark}{Remark}[section]
\newtheorem{example}{Example}[section]
\newtheorem{convention}{Convention}[section]
\newtheorem*{convention*}{Convention}

\providecommand{\keywords}[1]
{
 \small	
 \textbf{\textit{Keywords: }} #1
}
\date{}
\begin{document}
\title{Qualitative reasoning in a~two-layered framework\thanks{This is a~postprint version of the following paper --- \href{https://doi.org/10.1016/j.ijar.2022.12.011}{doi: 10.1016/j.ijar.2022.12.011}.}}
\author[1]{Marta B\'{\i}lkov\'{a}}
\author[2]{Sabine Frittella}
\author[2]{Daniil Kozhemiachenko}
\author[3]{Ondrej Majer}
\affil[1]{The Czech Academy of Sciences, Institute of Computer Science, Czech Republic}
\affil[2]{INSA Centre Val de Loire, Univ.\ Orl\'{e}ans, LIFO EA 4022, France}
\affil[3]{The Czech Academy of Sciences, Institute of Philosophy, Czech Republic}
\maketitle
\begin{abstract}
The reasoning with qualitative uncertainty measures involves comparative statements about events in terms of their likeliness without ne\-ces\-sarily assigning an exact numerical value to these events.

The paper is divided into two parts. In the first part, we formalise reasoning with the qualitative counterparts of capacities, belief functions, and probabilities, within the framework of two-layered logics. Namely, we provide two-layered logics built over the classical propositional logic using a~unary belief modality $\Be$ that connects the inner layer to the outer one where the reasoning is formalised by means of G\"{o}del logic. We design their Hilbert-style axiomatisations and prove their completeness. In the second part, we discuss the paraconsistent generalisations of the logics for qualitative uncertainty that take into account the case of the available information being contradictory or inconclusive.

\end{abstract}
\section{Introduction}
The main objective of this paper is to provide a~logical framework for reasoning with uncertain information, where uncertainty in events is expressed qualitatively rather than quantitatively. This approach has been explored in the case of (qualitative counterparts of) probability, but much less so in the case of weaker notions of uncertainty, creating a~gap we aim to fill with this paper. The most general kind of uncertainty measure usually considered in the literature is \emph{capacity} ---~a function $\mu:2^W\rightarrow[0,1]$ on the powerset of a~set of events $W\neq\varnothing$ which is monotone w.r.t.\ $\subseteq$, non-trivial ($\mu(W)>\mu(\varnothing)$), and normalised: $\mu(\varnothing)=0$, $\mu(W)=1$ (cf.~\cite{Grabisch2009,Zhou2013}). We slightly generalise this notion by dropping the last condition and call the resulting function an \textit{uncertainty measure}.

The difference between normalised and non-normalised measures is related to the difference between closed and open-world assumptions and to the difference between normal and non-normal modal logics. In other words, an agent may not necessarily believe that they have access to the whole sample space. Thus, even though, a classical tautology $\phi$ is true in all states accessible by an agent, they are still not convinced therein to assign $1$ as $\phi$'s degree of certainty.

We will employ a~logical formalism consisting of two separate layers: the inner layer formalises reasoning about events, while the outer layer formalises reasoning about beliefs in the events, the uncertainty measure of choice acting as the interpretation of the belief modality acting on formulas describing events. We will use G\"{o}del many-valued (fuzzy) logic to formalise the qualitative reasoning about the resulting beliefs on the outer layer. Apart from reasoning about events classically, we also wish to consider cases where the actual information about events available to an agent is faulty in the sense that it can be incomplete or even contradictory. Thus, in the apparatus of two-layered logics, we will employ not only classical logic but also \emph{fuzzy} and \emph{paraconsistent} logics to account for this possibility and reason with non-classical information in a~\emph{non-trivial} way.

\paragraph{Weaker uncertainty measures} The most prominent examples of uncertainty measures are probability functions, however, they might be considered too strong to formalise people's beliefs.
Indeed, probabilities assume that the agents are able to express their certainty in a~statement with some number from $[0,1]$ and then conduct arithmetical operations with these numbers to obtain the probability assignments of complex events: add or subtract them, multiply them by some given constants, etc. This, however, is not always the case. Indeed, outside of scientific contexts (such as conducting medical tests, forecasting the weather, etc.), one can rarely give a~number that corresponds to their certainty in a~given statement. Furthermore, a~layperson might not necessarily know how to obtain values of complex events, although they are usually able to compare their certainty in those.

In~\cite{BilkovaFrittellaMajerNazari2020}, this aspect was partially tackled with a~logic of monotone coherent belief where not only the event description but also the reasoning was done in the language of the Belnap--Dunn logic ($\BD$). $\BD$, however, lacks the capacity to express statements such as ‘I think that the rain tomorrow is more likely than hailstorm’. Here the agent is a~layperson and does not know the precise probabilities of any event. Neither are they likely to say something like ‘I am $73\%$ sure that it is going to rain tomorrow’. On the other hand, one often can \emph{compare} their degree of certainty in two given statements or at least be able to say something like ‘I don't know whether rain or hailstorm is more likely tomorrow’.

For the classical case, one of such qualitative counterparts of uncertainty measures is \emph{qualitative probability}. One of the formal systems axiomatising qualitative probability is the modal logic proposed by Peter G\"{a}rdenfors in~\cite{Gardenfors1975}. It uses a~binary modality $\lesssim$ such that $\phi\lesssim \psi$ is interpreted as ‘$\phi$ is at most as likely as~$\psi$’. Furthermore, G\"{a}rdenfors' calculus allows for nesting of $\lesssim$ which, in turn, leads to formulas of the form
\[(\tau\lesssim\phi)\lesssim(\chi\lesssim\psi)\]
to be read as ‘it is at most as likely that $\tau$ is at most as likely as $\phi$ as that $\chi$ is at most as likely as~$\psi$’. These statements, however, are quite cumbersome and rarely used in natural language reasoning when people are comparing their certainty in different statements.

This is why, in this paper, we will present a~two-layered logic based on G\"{o}del logic expanded with co-implication $\biG$ that formalises qualitative reasoning with uncertainty measures. We will also show how to extend it with additional axioms corresponding to the qualitative counterparts of capacities, belief functions, and probabilities. For the last case, we will also show that our logic, in fact, interprets a~certain fragment of the G\"{a}rdenfors' calculus where the $\lesssim$ does not nest.

\paragraph{Two-layered logics for uncertainty} Two-layered logics consist of two languages: the one that describes the events such as ‘if it rains, then there should be puddles on the road’ (inner layer), and the one wherein the reasoning about beliefs in these events occurs. The uncertainty measure then acts as a~non-nesting modality $\mathtt{M}$. Thus, the formulas of the outer language are built from the atoms of the form $\mathtt{M}\phi$ with $\phi$ being an inner-layer formula.

One of the first formalisations of a~logic for reasoning about classical probability is provided in~\cite{FaginHalpernMegiddo1990,FaginHalpern1991ComputationalIntelligence}. There, the events are described using classical propositional logic. Then, the reasoning happens with so-called \emph{weight formulas} composed of the following constructions
\[\sum\limits_{i=1}^{n}a_i\cdot\mathtt{w}(\phi_i)\geqslant c.\]
Here, $a_i,c\in\mathbb{Z}$, $\phi_i$ is a~propositional formula, and $\mathtt{w}$ a~\emph{weight} modality denoting the \emph{probability} of the formula.

Although, the formalisation in~\cite{FaginHalpern1991ComputationalIntelligence} is not a~two-layered logic \emph{per se}, it can be equivalently translated into one. This equivalence is demonstrated in~\cite{BaldiCintulaNoguera2020}. There, the outer layer is expressed in the language of \L{}ukasiewicz logic, as was first proposed in~\cite{HajekGodoEsteva1995}. This removes the need of operating with the Boolean combinations of inequalities directly and allows for a~finite and straightforward axiomatisation of the whole framework.

The use of many-valued logics (we refer the reader to~\cite{Hajek1998} and~\cite{CintulaHajekNoguera2011HandbookofMFL1,CintulaHajekNoguera2011HandbookofMFL2,CintulaFermuellerNoguera2015HandbookofMFL3} for basic exposition to and further details on many-valued and fuzzy logics) on the outer layer ties into the existing tradition of using these for reasoning about vagueness~\cite{Smith2008}, beliefs~\cite{GodoHajekEsteva2003,TouaziCayrolDubois2015}, and uncertainty~\cite{DuboisPradeSchockaert2017}. Most of the previous work, however, employs logics that are capable of expressing a~certain amount of arithmetic (usually, addition and subtraction). 
As this paper concentrates on the qualitative notion of uncertainty, we will be using G\"{o}del logic and its expansions that can only \emph{compare values} of formulas but not conduct arithmetic operations with them. Still, we will show that this is enough to have a~logic that can formalise reasoning about probabilities \emph{qualitatively}.

\paragraph{Two-dimensional treatment of inconsistent information}
If we recall the understanding of uncertainty, outlined in the first paragraphs of this section, it is clear that classical logic (unsurprisingly) is not well suited to reasoning with non-classical information. Indeed, an arbitrary statement follows classically from a~contradiction. This, however, is counter-intuitive for there may be no connection between those. Likewise, even if one does not have any evidence at all whether $\phi$ is true or false, $\phi\vee{\sim}\phi$ is valid, whence, always true. In other words, classical logic cannot reason with non-classical, and most importantly, contradictory information \emph{non-trivially}.

Thus, the goal of our project is to develop \emph{paraconsistent} two-dimensional logics and apply them to the reasoning about uncertainty. Namely, we propose logics built upon or inspired by Belnap and Dunn's ‘useful four-valued logic’ (alias, Belnap--Dunn logic, $\BD$ or first-degree entailment)~\cite{Dunn1976,Belnap2019,Belnap1977fourvalued}.

The main idea of $\BD$ is to treat positive and negative information regarding a~given statement\footnote{Originally, the values denote what a~computer or a~database might have been told regarding a~statement.} $\phi$ independently. This gives rise to four values comprising the lattice $\mathbf{4}$ depicted on fig.~\ref{fig:BDsquare}. There, $t$ stands for ‘the information only supports the truth of $\phi$’; $f$ for ‘the information only supports the falsity of $\phi$’; $b$ for ‘the information supports both the truth and falsity of $\phi$’ (i.e., the information is contradictory); $n$ for ‘there is no information regarding $\phi$’. The upwards order $x\leq_\mathbf{4}y$ is construed as ‘$x$ is at least as true and at most as false as $y$’.

Should we wish to account for the possibility of the non-classical information to be graded, the original lattice $\mathbf{4}$ is extended to $[0,1]^{\Join}$ (fig.~\ref{fig:NSProb}). Two-layered logics axiomatising paraconsistent probabilities~\cite{KleinMajerRad2021}, as well as some other uncertainty measures including belief functions, were presented in~\cite{BilkovaFrittellaMajerNazari2020,BilkovaFrittellaKozhemiachenkoMajerNazari2023APAL}. There, the events are described by $\BD$ propositions which represents the possibility of our sources to provide non-classical information. The reasoning on the outer layer is formalised using a~paraconsistent expansion of \L{}ukasiewicz logic with the Belnap--Dunn negation $\neg$, and the uncertainty measure, in its turn, is a~non-standard probability or belief function.
\begin{figure}[h]
\begin{minipage}[h]{0.45\linewidth}
\begin{center}
\begin{tikzpicture}[>=stealth,relative]
\node (U1) at (0,-1.2) {$f$};
\node (U2) at (-1.2,0) {$n$};
\node (U3) at (1.2,0) {$b$};
\node (U4) at (0,1.2) {$t$};
\path[-,draw] (U1) to (U2);
\path[-,draw] (U1) to (U3);
\path[-,draw] (U2) to (U4);
\path[-,draw] (U3) to (U4);
\end{tikzpicture}
\caption{De Morgan lattice $\mathbf{4}$}
\label{fig:BDsquare}
\end{center}
\end{minipage}
\hfill
\begin{minipage}[h]{0.45\linewidth}
\begin{center}
\begin{tikzpicture}[-,>=stealth,shorten >=0.5pt,auto,node distance=1.2cm,thin,
	main node/.style={circle,draw,font=\sffamily\normalsize},]
\node[main node][label=left:{$(0,0)$}] (1) {};
\node[main node][label={$(1,0)$}] (2) [above right of=1] {};
\node[main node][label=below:{$(0,1)$}] (3) [below right of=1] {};
\node[main node][label=right:{$(1,1)$}] (4) [above right of=3] {};
	
\path[every node/.style={font=\sffamily\small}]
(1) edge (2)
edge (3)
(2) edge (4)
(3) edge (4);
\path[dotted]
(2) edge (3)
(1) edge (4);
\end{tikzpicture}
\end{center}
\caption{$[0,1]^{\Join}$ --- continuous extension of~$\mathbf{4}$. $(x,y)\!\leq_{[0,1]^{\Join}}\!(x',y')$ iff $x\!\leq\!x'$ and $y\!\geq\!y'$}
\label{fig:NSProb}
\end{minipage}
\end{figure}
\paragraph{Plan of the paper}
The remainder of the paper is structured as follows. In Section~\ref{sec:preliminaries}, we formulate an expansion of G\"{o}del logic with the coimplication $\coimplies$ dubbed $\biG$. In particular, we define its semantics and provide axiomatisation; we also recall that $\coimplies$ and the Baaz delta operator $\triangle$ are interdefinable and show how to axiomatise $\biG$ with $\triangle$ instead of $\coimplies$.

In Section~\ref{sec:QPC}, we show how to axiomatise reasoning with different qualitative uncertainty measures. We begin with the logic we dub $\QG$ that formalises the reasoning with arbitrary uncertainty measures using G\"{o}del logic. We then consider axiomatic extensions of $\QG$ that formalise qualitative reasoning with capacities and belief functions. Finally, we provide $\QPG$ --- the two-layered logic of qualitative probabilities. We also recall G\"{a}rdenfors' logic of qualitative probability $\QP$ and establish a~faithful embedding of ‘simple inequality formulas’ (i.e., formulas constructed as Boolean combinations of $\chi\lesssim\chi'$ formulas of $\QP$ without nesting of $\lesssim$) into $\QPG$.

In Section~\ref{sec:MCB}, we recall the Belnap--Dunn logic $\BD$ and the paraconsistent expansions of $\biG$. We use them to develop two paraconsistent expansions of $\QG$. We define semantics, construct calculi and prove their completeness. We use these logics to formalise comparisons of certainty in (potentially) contradictory or incomplete information.

Finally, in Section~\ref{sec:conclusion}, we recapitulate the results presented in the paper and discuss future work.
\section{Preliminaries: bi-G\"{o}del logic\label{sec:preliminaries}}
G\"{o}del logic is an infinite-valued propositional logic, with its standard algebraic semantics being based on the full $[0,1]$ interval, where $1$ is the designated value. The truth values are (densely) ordered, and, together with the semantics of G\"{o}del implication, this makes G\"{o}del logic suitable for formalising comparisons. G\"{o}del logic is one of the three basic t-norm-based fuzzy logics, and it is also closely related to intuitionistic logic: it is the logic of linearly ordered Heyting algebras and can also be characterised as the logic of linearly ordered intuitionistic Kripke structures, and axiomatized by extending the intuitionistic logic with the axiom of prelinearity. A~more detailed exposition of G\"{o}del logics can be found e.g., in~\cite{CintulaHajekNoguera2011HandbookofMFL2}.

In this subsection, we formulate G\"{o}del logic expanded with a~co-implication connective $\coimplies$ and refer to it as $\biG$ (bi-G\"{o}del logic or symmetric G\"{o}del logic in the terminology of~\cite{GrigoliaKiseliovaOdisharia2016}), as it can naturally be obtained by extending the bi-intuitionistic logic with the axioms of prelinearity. Note that instead of $\coimplies$, one could add the Baaz delta operator $\triangle$ and obtain the logic that is expressively equivalent to $\biG$ (cf.~Remark~\ref{rem:trianglecoimplies}). In the following definition of bi-G\"{o}del algebras, we leave both $\triangle$ and $\coimplies$ as it will facilitate the formalisation of comparative belief statements.
\begin{definition}\label{def:bi-G_algebra}
The bi-G\"{o}del algebra $[0,1]_{\mathsf{G}}=\langle[0,1],0,1,\wedge_\mathsf{G},\vee_\mathsf{G},\rightarrow_{\mathsf{G}},\coimplies,\sim_\mathsf{G},\triangle_\mathsf{G}\rangle$ is defined as follows: for all $a,b\in[0,1]$, $\wedge_\mathsf{G}$ and $\vee_\mathsf{G}$ are given by $a\wedge_\mathsf{G}b\coloneqq\min(a,b)$, $a\vee_\mathsf{G}b\coloneqq\max(a,b)$. The remaining operations are defined below:
\begin{align*}
a\rightarrow_\mathsf{G}b&=
\begin{cases}
1,\text{ if }a\leq b\\
b\text{ else}
\end{cases}
&
a\coimplies_\mathsf{G}b&=
\begin{cases}
0,\text{ if }a\leq b\\
a\text{ else}
\end{cases}
&
{\sim_\mathsf{G}}a&=
\begin{cases}
0,\text{ if }a>0\\
1\text{ else}
\end{cases}
&
\triangle_\mathsf{G}a&=
\begin{cases}
0,\text{ if }a<1\\
1\text{ else}
\end{cases}
\end{align*}
\end{definition}
\begin{definition}[Language and semantics of $\biG$]\label{def:biGlogic}
We set $\Prop$ to be a~countable set of propositional variables and consider the following language $\LbiG$.
\[\LbiG:~\phi\coloneqq p\in\Prop\mid{\sim}\phi\mid(\phi\wedge\phi)\mid(\phi\vee\phi)\mid(\phi\rightarrow\phi)\mid(\phi\coimplies\phi)\]
We let $e:\Prop\rightarrow[0,1]$. Using bi-G\"{o}del operations from Definition~\ref{def:bi-G_algebra}, $e$ is extended to complex formulas in the expected manner: $e(\phi\circ\phi')\!=\!e(\phi)\circ_\mathsf{G}e(\phi')$. We say that $\phi$ is \emph{valid} iff $e(\phi)\!=\!1$ for any~$e$. Furthermore, we define the entailment:
$$\Gamma\models_{\biG}\chi\text{ iff }\inf\{e(\phi):\phi\in\Gamma\}\leq e(\chi)\text{ for any }e$$
\end{definition}
\begin{convention}[Notational conventions]
We will further use the following shorthands.
\begin{align*}
\top&\coloneqq p\rightarrow p&\bot&\coloneqq p\coimplies p
\end{align*}
Note that
\begin{align*}
e(\top)&=1&e(\bot)&=0
\end{align*}
Finally, we write $e[\Gamma]=x$ iff $\inf\{e(\phi):\phi\in\Gamma\}=x$.
\end{convention}
\begin{remark}\label{rem:trianglecoimplies}
Observe that $\coimplies$ and $\triangle$ are interdefinable:
\begin{align*}
\triangle\phi&\coloneqq\top\coimplies(\top\coimplies\phi)&\phi\coimplies\phi'&\coloneqq\phi\wedge{\sim}\triangle(\phi\rightarrow\phi')
\end{align*}
\end{remark}

G\"{o}del logic can be interpreted as a~logic of comparative truth, for the truth of a~given formula depends only on the relative order of the values of its variables, not the values themselves. It is best seen in the semantics $\rightarrow$.

Indeed, $\phi\rightarrow\phi'$ is true (has value $1$) iff the value of $\phi$ is less or equal to that of $\phi'$. In simpler words, for an implication to be true, the value cannot \emph{decrease} from the antecedent to the consequent. On the other hand, if $\phi\rightarrow\phi'$ is not true, then we can safely assume that its truth degree is not smaller than that of $\phi'$. In what follows, we will use the defined connective $\triangle$ rather than $\coimplies$ to express statements of the form ‘I am more certain in $\phi$ than in $\phi'$’ (cf.~Examples~\ref{example:beliefcomparison} and~\ref{example:naturallanguage} for more details) because it allows for more concise and straightforward presentation. On the other hand, the use of $\coimplies$ in the calculus simplifies the axiomatisations of the paraconsistent expansions of $\biG$.

Let us now present the Hilbert-style calculus for $\biG$ that we dub $\HbiG$.
\begin{definition}[$\HbiG$]\label{def:HHB}
The calculus has the following axiom schemas and rules (for any $\phi$, $\chi$,~$\psi$):
\begin{enumerate}
\item $(\phi\rightarrow\chi)\rightarrow((\chi\rightarrow\psi)\rightarrow(\phi\rightarrow\psi))$
\item $\phi\rightarrow(\phi\vee\chi)$; $\chi\rightarrow(\phi\vee\chi)$; $(\phi\rightarrow\psi)\rightarrow((\chi\rightarrow\psi)\rightarrow((\phi\vee\chi)\rightarrow\psi))$
\item $(\phi\wedge\chi)\rightarrow\phi$; $(\phi\wedge\chi)\rightarrow\chi$; $(\phi\rightarrow\chi)\rightarrow((\phi\rightarrow\psi)\rightarrow(\phi\rightarrow(\chi\wedge\psi)))$
\item $(\phi\rightarrow(\chi\rightarrow\psi))\rightarrow((\phi\wedge\chi)\rightarrow\psi)$; $((\phi\wedge\chi)\rightarrow\psi)\rightarrow(\phi\rightarrow(\chi\rightarrow\psi))$
\item $(\phi\rightarrow\chi)\rightarrow({\sim}\chi\rightarrow{\sim}\phi)$
\item $(\phi\coimplies\chi)\rightarrow(\top\coimplies(\phi\rightarrow\chi))$; ${\sim}(\phi\coimplies\chi)\rightarrow(\phi\rightarrow\chi)$
\item $\phi\rightarrow(\chi\vee(\phi\coimplies\chi))$; $((\phi\coimplies\chi)\coimplies\psi)\rightarrow(\phi\coimplies(\chi\vee\psi))$
\item $(\phi\rightarrow\chi)\vee(\chi\rightarrow\phi)$; $\top\coimplies((\phi\coimplies\chi)\wedge(\chi\coimplies\phi))$
\item[MP] $\dfrac{\phi\quad\phi\rightarrow\chi}{\chi}$
\item[nec] $\dfrac{\vdash\phi}{\vdash{\sim}(\top\coimplies\phi)}$
\end{enumerate}
We say that $\phi$ is \emph{derived} from $\Gamma$ iff there exists a~finite sequence of formulas $\phi_1$, \ldots, $\phi_n=\phi$ each of which is either an instantiation of an axiom schema, a~member of $\Gamma$, or obtained from previous ones by a~rule. If $\Gamma=\varnothing$, we say that $\phi$ is \emph{proved}.
\end{definition}

Observe that in the definition above, the first five groups of axioms formalise intuitionistic logic, adding axioms 6 and 7 produces the axiomatisation of Heyting--Brouwer (bi-intuitionistic) logic~\cite{Rauszer1977,Rauszer1980phd}. Finally, axiom 8 stands for the linearity conditions for $\rightarrow$ and $\coimplies$.
\begin{remark}
The axiomatisation of $\Gtriangle$ (G\"{o}del logic with Baaz delta) can be easily obtained from Definition~\ref{def:HHB}. Instead of the axioms and rules with $\coimplies$, one should add the following axioms for $\triangle$ from~\cite{Baaz1996}
\begin{align*}
\triangle\phi\vee{\sim}\triangle\phi&&\triangle(\phi\rightarrow\chi)\rightarrow(\triangle\phi\rightarrow\triangle\chi)&&\triangle(\phi\vee\chi)\rightarrow(\triangle\phi\vee\triangle\chi)&&\triangle\phi\rightarrow\phi&&\triangle\phi\rightarrow\triangle\triangle\phi
\end{align*}
as well as the $\triangle$ necessitation rule $\dfrac{\vdash\phi}{\vdash\triangle\phi}$.
\end{remark}
\section{Qualitative reasoning about uncertainty\label{sec:QPC}}
In many contexts, an agent may not necessarily be able to give a~precise number that corresponds to their degree of certainty in a~given proposition. Indeed, ‘we are 70\% certain that it is going to rain tomorrow’ is relatively natural when giving a~weather forecast based on some evidence as well as on previous statistical data. On the other hand, if two different people (say, Paula and Quinn) claim that a~recently found dog belongs to themselves only, a~statement such as ‘I~am 65\% confident that the dog is Paula's’ sounds bizarre.

Nevertheless, in such contexts, there is a~chance (if Paula, Quinn, or both present some relevant evidence) for a~reasoning agent to compare their claims and pick a~more compelling one. In other words, even though, we are not really able to measure our certainty in a~given statement \emph{numerically}, we can compare evidence for it with evidence for another one.

Qualitative probability was extensively studied, starting already in the work of de Finetti~\cite{deFinetti1937}. Later, Kraft, Pratt and Seidenberg in their seminal paper~\cite{KraftPrattSeidenberg1959} axiomatised the qualitative characterisation of probability measures in terms of qualitative orderings of sets of events. More formally, a~measure $\mu$ on a~set of events $\P(W)$ agrees with a~qualitative ordering $\preccurlyeq$ on the same set iff
\[\forall X,Y\subseteq W:X\preccurlyeq Y\Leftrightarrow\mu(X)\leq\mu(Y)\]

The qualitative counterparts of capacities and belief functions were studied as well (see e.g., \cite{WongYaoBollmannBurger1991,WongYaoBollmann1992}). We summarise the results considering various uncertainty measures in the following theorem.
\begin{theorem}[Qualitative characterisations of uncertainty measures]\label{theorem:qualitativecharacterisations}
Let $W\neq\varnothing$ and let further $\preccurlyeq$ be a~linear preorder on $2^W$. Consider the following conditions on $\preccurlyeq$ for all $X,Y,Z\subseteq W$.
\begin{description}
\item[Q1] $\varnothing\preccurlyeq X\preccurlyeq W$.
\item[Q2] $\varnothing\prec W$.
\item[Q3] If $X\subseteq Y$, then $X\preccurlyeq Y$.
\item[PM] If $X\subsetneq Y$, $X\prec Y$, and $Y\cap Z=\varnothing$, then $X\cup Z\prec Y\cup Z$.
\item[$\mathbf{KPS}_m$] For any $m\in\mathbb{N}$ and all $X_i,Y_i\subseteq W$ ($i\in\{0,\ldots,m\}$), it holds that if $X_j\preccurlyeq Y_j$ for all $j<m$ and if any $w\in W$ belongs to as many $X_i$'s as $Y_i$'s, then $X_m\succcurlyeq Y_m$.
\end{description}
Then, it holds that
\begin{enumerate}
\item the counterparts of $\preccurlyeq$ are \emph{capacities} iff $\preccurlyeq$ satisfies $\mathbf{Q1}$--$\mathbf{Q3}$;
\item the counterparts of $\preccurlyeq$ are \emph{belief functions} iff $\preccurlyeq$ satisfies $\mathbf{Q1}$--$\mathbf{Q3}$ and $\mathbf{PM}$;
\item the counterparts of $\preccurlyeq$ are \emph{probability measures} iff $\preccurlyeq$ satisfies $\mathbf{Q1}$--$\mathbf{Q3}$ and $\mathbf{KPS}_m$. 
\end{enumerate}
\end{theorem}
The list in Theorem~\ref{theorem:qualitativecharacterisations} is compiled from different sources, whence its obvious redundancy: $\mathbf{Q3}$ entails $\mathbf{Q1}$, and moreover, $\mathbf{Q1}$, $\mathbf{Q2}$, and $\mathbf{KPS}_m$ entail $\mathbf{Q3}$. We decided to leave the redundant conditions to make the presentation more uniform.

We will also consider functions, which correspond to non-normalised capacities, i.e., they only satisfy monotonicity and nontriviality. We call them \emph{uncertainty measures}. Note that just as capacities, uncertainty measures are too counterparts of $\preccurlyeq$ satisfying $\mathbf{Q1}$--$\mathbf{Q3}$. 


Moreover, since any uncertainty measure $\mu:2^W\rightarrow[0,1]$ gives rise to a~linear preorder $\preccurlyeq_\mu$ on $2^W$, the original conditions $\mathbf{PM}$ (‘partial monotonicity’, in the terminology of~\cite{WongYaoBollmannBurger1991,WongYaoBollmann1992}) and $\mathbf{KPS}_m$ (Kraft--Pratt--Seidenberg conditions) can be reformulated in terms of measures. For any uncertainty measure $\mu$, its qualitative counterpart is an ordering corresponding to a~belief function iff $\mu\mathbf{PM}$ holds and an ordering corresponding to a~probability measure iff $\mu\mathsf{KPS}_m$ holds.
\begin{description}
\item[$\mu\mathbf{PM}$:]\label{item:muPM} If $\mu(X)<\mu(Y)$, $X\subsetneq Y$, and $Y\cap Z=\varnothing$, then $\mu(X\cup Z)<\mu(Y\cup Z)$.
\item[$\mu\mathsf{KPS}_m$:]\label{item:muKPSm} For any $m\in\mathbb{N}$ and all $X_i,Y_i\subseteq W$ ($i\in\{0,\ldots,m\}$), it holds that if $\mu(X_j)\leq\mu(Y_j)$ for all $j<m$ and if any $w\in W$ belongs to as many $X_i$'s as $Y_i$'s, then $\mu(X_m)\geq\mu(Y_m)$.
\end{description}
In what follows, we present two-layered logics that formalise qualitative reasoning with relative likelihood orderings corresponding to uncertainty measures, capacities, belief functions and probabilities.
\subsection{Qualitative uncertainty in a~two-layered framework}
In this section, we formulate two-layered modal logic $\QG=\langle\CPL,\{\Be\},\biG\rangle$. Here, we represent the events with the classical propositional logic but reason with the beliefs concerning these events using $\biG$~--- G\"{o}del logic with co-implication. This will allow us to compare likelihoods of different events. The modality connecting two layers is denoted with $\Be$. Formulas of the form $\Be\phi$ (with $\phi\in\LCPL$\footnote{We use a~minimalistic $\{{\sim},\wedge\}$ language with all other connectives being introduced via definitions.}) are interpreted as ‘the agent believes in $\phi$’ with its value being understood as the degree of truth (cf.~\cite{CintulaFermuellerNoguera2021SEP} and the literature referenced therein for a~more detailed discussion of truth degrees). Recall, however, that in G\"{o}del logic a~formula is (fully) true (i.e., has value $1$) in virtue of the \emph{relative order} of the values of its variables, not their \emph{numerical values}, and that this order corresponds to the G\"{o}del implication.

Other connectives can be understood in an expected manner: $\wedge$ stands for ‘and’; $\vee$ for ‘or’; $\coimplies$ for ‘excludes’\footnote{Cf.~\cite{Wansing2008} for further details on the interpretation of $\coimplies$ as ‘excludes’.}; $\triangle$ for ‘it is (completely) true that’; $\sim$ for ‘it is (completely) false that’. We argue that the G\"{o}del implication is the most natural implication for the qualitative and comparative contexts as opposed to the \L{}ukasiewicz implication that is used in \emph{quantitative} representations of uncertainty (cf., e.g.~\cite{HajekGodoEsteva1995}) for two reasons. First, it residuates $\wedge$ which is the natural comparative conjunction. Second, even though, the \L{}ukasiewicz implication respects the order on $[0,1]$ as well, it defines the truncated sum and subtraction connectives $\oplus$ and $\ominus$ which makes it inherently quantitative. However, in the qualitative context, an agent might not be able to give values to their certainty in given statements, whence they cannot meaningfully add up or subtract these values. On the other hand, the G\"{o}del implication together with the coimplication (or, alternatively, Baaz delta operator) is both sufficient to express the order on $[0,1]$ (both strict and non-strict ones) and does not have arithmetic operations as its by-product. Third, as we have already mentioned, if an implicative formula is not true, one can still safely assume that its truth degree is not lower than that of the succedent. Moreover, the use of the bi-G\"{o}del logic allows for the proof of strong completeness of its two-layered expansions while the two-layered expansions of the \L{}ukasiewicz logic are only weakly complete.

Note that the two-layered nature of $\QG$ prohibits the nesting of $\Be$ modality. In other words, statements such as ‘it is more likely that $\phi$ is more likely than $\chi$ being more likely than $\psi$’ cannot be expressed. We will show, however, that there is a~faithful translation that preserves the validity of propositional combinations of formulas $\chi\!\lesssim\!\chi'$ (with $\chi,\chi'\!\in\!\LCPL$).

The present paper is by no means the first one on the representation of qualitative uncertainty. In particular, the qualitative logic of possibility is introduced in~\cite{FarinasdelCerroHerzig1991} where the authors follow the G\"{a}rdenfors' approach and extend classical propositional logic with axioms of qualitative possibility formulated with $\lesssim$. The logic is then translated into the quantitative possibility logic and shown to preserve validity between quantitative and qualitative models.

Another approach to the logic of qualitative possibility is presented in~\cite[Section~2.9]{Halpern2017}. There, Halpern introduces a~qualitative notion of relative likelihood and shows that every order based on a~possibility measure is in fact a~relative likelihood. In other words, his axiomatization of qualitative possibility gives necessary conditions for the existence of a~compatible possibility measure, he does not discuss the question if these conditions are also necessary. 

An alternative approach to the axiomatisation of qualitative probabilities is presented in~\cite{DelgrandeRenne2015,DelgrandeRenneSack2019}. The authors introduce an additional connective $\oplus$ that defines qualitative probability on sequences of formulas
\begin{align*}
\bigoplus\limits^{n}_{i=1}\phi_i\lesssim\bigoplus\limits^{n}_{i=1}\psi_i&\text{ iff }\sum\limits^{n}_{i=1}\mathtt{p}(\|\phi_i\|)\leq\sum\limits^{n}_{i=1}\mathtt{p}(\|\psi_i\|)\tag{$\mathtt{p}$ is a~probability measure}
\end{align*}
This allows for a~finite axiomatisation in contrast to the G\"{a}rdenfors' $\QP$ (cf.~Definition~\ref{def:HQP}) of the logic as $\oplus$ can be used to express additivity. On the other hand, $\oplus$ makes the logic hybrid rather than purely qualitative. Moreover, the existence of a~quantitative measure is assumed rather than derived from the qualitative relation as it is done both traditionally and in the present paper.

Yet another treatment of qualitative probabilities inspired by~\cite{Savage1972} is presented in~\cite{BaldiHosni2021}. The authors devise a~sequence of finitely axiomatised calculi that approximate qualitative reasoning with probability measures.

One of the most important distinctions of our approach from the ones discussed above is that we use \emph{unary} belief modalities $\Be\phi$ whose values are understood as truth degrees of the~statement ‘agent believes that $\phi$’ while traditionally, a~\emph{binary} modality $\lesssim$ (‘at least as likely as’) is used. At the first glance, the use of a~binary modality seems to be more justified and straightforward when one deals with qualitative or comparative contexts. However, as one can see from Theorem~\ref{theorem:qualitativecharacterisations}, $\lesssim$ cannot distinguish between normalised measures (i.e., the ones where $\mu(\|\bot\|)=0$ and $\mu(\|\top\|)=1$) and the non-normalised ones. In a~sense, binary modalities cannot express statements such as ‘the agent has a~positive belief in $p$’. Later (cf.~Remark~\ref{rem:nonrepresentableproperties}), we will also see some formulas expressing some natural properties of measures that cannot be expressed using~$\lesssim$.

Moreover, we claim that our approach is purely qualitative in contrast to those of Delgrande--Renne--Sack in the following sense. First, we do not \emph{a~priori} assume that the measure $\mu$ on the frame is in fact a~belief function, a~probability, etc.\ but merely that the order $\preccurlyeq$ corresponding to $\mu$ conforms to the needed conditions from Theorem~\ref{theorem:qualitativecharacterisations}. Second, our language does not express the quantitative axioms of any uncertainty measure in contrast to that of~\cite{DelgrandeRenne2015,DelgrandeRenneSack2019}. In this, we follow the approach of G\"{a}rdenfors~\cite{Gardenfors1975}: we begin with frames whereon a~measure is defined and then axiomatise the qualitative conditions corresponding to this measure.

Let us now introduce $\QG$ in a~formal manner. We are first building our qualitative framework correspondingly to generic uncertainty measures. Then, we add conditions characterising qualitative counterparts of stronger measures which are usually discussed in the literature: in particular, capacities, belief functions, and probability functions.
\begin{definition}[Language and semantics of $\QG$]\label{def:QGsemantics}
We let $\phi\in\LCPL$ and define the following grammar.
\[\LQG:~\alpha\coloneqq\Be\phi\mid{\sim}\alpha\mid(\alpha\wedge\alpha')\mid(\alpha\vee\alpha')\mid(\alpha\rightarrow\alpha)\mid(\alpha\coimplies\alpha')\]
An \emph{uncertainty frame} (or simply, a~frame) is a~tuple $\mathbb{F}=\langle W,\mu\rangle$. Here, $W\neq\varnothing$ and $\mu$ is an uncertainty measure on $W$.

A \emph{$\QG$ model} is a~tuple $\mathfrak{M}=\langle W,v,\mu,e\rangle$ with $\langle W,\mu\rangle$ being a~frame and $v:\Prop\rightarrow2^W$ being extended to $\mathfrak{M},x\vDash\phi$ for $x\in W$ and $\phi\in\LCPL$ as follows.
\begin{itemize}
\item $\mathfrak{M},x\vDash p$ iff $x\in v(p)$.
\item $\mathfrak{M},x\vDash{\sim}\phi$ iff $\mathfrak{M},x\nvDash\phi$.
\item $\mathfrak{M},x\vDash\phi\wedge\phi'$ iff $\mathfrak{M},x\vDash\phi$ and $\mathfrak{M},x\vDash\phi'$.
\end{itemize}
$e$ is a~bi-G\"{o}del valuation (cf.~Definition~\ref{def:biGlogic}) s.t.\ $e(\Be\phi)=\mu(\|\phi\|)$ with $\|\phi\|=\{x:\mathfrak{M},x\vDash\phi\}$.

$\QG$ entailment is defined in the same manner as that of $\biG$:
$$\Xi\models_{\QG}\alpha\text{ iff }\inf\{e(\gamma):\gamma\in\Xi\}\leq e(\alpha)\text{ for any }e\text{ induced by an uncertainty measure}$$

Finally, $\alpha\in\LQG$ is valid on $\mathbb{F}$ ($\mathbb{F}\models\alpha$) iff $e(\alpha)=1$ for any $e$ and $v$ defined on~$\mathbb{F}$.
\end{definition}
\begin{remark}\label{rem:QGnoncompositionality}
One can notice that $\Be$ is \emph{non-compositional} in the sense that neither $\Be(t\wedge r)\leftrightarrow(\Be t\wedge\Be r)$, nor $\Be(t\vee r)\leftrightarrow(\Be t\vee\Be r)$ are $\QG$ valid\footnote{In fact, it is easy to see that there is no general definition of $\Be\phi(p_1,\ldots,p_n)$ via $\Be p_i$'s.}. However, it can be argued~\cite{Dubois2008} that belief should not be compositional. In fact, all standard uncertainty measures are non-compositional, so it is expected that belief based on these measures is non-compositional too. It is even more true for the case when the truth value of a~given belief statement is graded.

Indeed, let $t$ stand for ‘the temperature in the cellar is $26^\circ\mathtt{C}$’ and $r$ for ‘it is raining outside right now’. The agent is in the cellar right now but there is no thermometer and no windows either. The agent does not feel very cold or hot, so $t$ seems reasonable (say, $v(\Be t)=0.7$); half an hour ago it was cloudy and wet, so the rain is not at all excluded (say, $v(\Be r)=0.5$). However, $t$ and $r$ are not entirely independent, thus, it is hardly possible to precisely determine the degree of certainty in either $\Be(t\vee r)$ or $\Be(t\wedge r)$.
\end{remark}
Let us provide an example of how to formalise statements concerning comparisons of degrees of certainty.
\begin{example}[Comparing certainty in $\QG$]\label{example:beliefcomparison}
Assume that two people, Paula and Quinn, come to you and say that a~recently found stray dog belongs to them (and not to the other person). Thus, we have two events: $p\wedge{\sim}q$ (the dog belongs to Paula but not to Quinn) and ${\sim}p\wedge q$ (vice versa). Now, assume further that you trust Paula more than you trust Quinn for some reason. Thus, the following statement should be true
\begin{quote}
\textit{I am more certain that the dog belongs to Paula than to Quinn.
}\end{quote}
We formalise it as follows
\begin{align}
\underbrace{\triangle(\Be({\sim}p\wedge q)\rightarrow\Be(p\wedge{\sim}q))}_{L}\wedge\underbrace{{\sim}\triangle(\Be(p\wedge{\sim}q)\rightarrow\Be({\sim}p\wedge q))}_{R}\label{equ:comparisonexample}
\end{align}
$L$ part establishes that your certainty in $p\wedge{\sim}q$ is at least as high as your certainty in ${\sim}p\wedge q$, while $R$ part means that the converse is not the case. But note that $\triangle(\alpha\rightarrow\alpha')\vee{\sim}\triangle(\alpha'\rightarrow\alpha)$ is a~$\biG$ valid formula, whence~\eqref{equ:comparisonexample} can be simplified to
\begin{align}
{\sim}\triangle(\Be(p\wedge{\sim}q)\rightarrow\Be({\sim}p\wedge q))\label{equ:comparisonexamplesimplified}
\end{align}
\end{example}
It is a~useful property of the bi-G\"odel logic that $\phi(p_1,\ldots,p_n)\in\LbiG$ is valid iff it is valid for all valuations that range over $\left\{0,\frac{1}{n+1},\ldots,\frac{n}{n+1},1\right\}$~\cite[Theorem~35]{Haehnle2001HBPL}. We can use this fact to provide a~natural language interpretation of formulas avoiding reference to numerical values altogether.
\begin{example}\label{example:naturallanguage}
Recall \eqref{equ:comparisonexamplesimplified}. Note that there are two $\LQG$ atoms: $\Be(p\wedge{\sim}q)$ and $\Be({\sim}p\wedge q)$. Thus, we need four (numerical) values corresponding to the ordered set: $\left\{0,\frac{1}{3},\frac{2}{3},1\right\}$. We can associate them with the following subjective values: \emph{certainly false}, \emph{unlikely}, \emph{likely}, \emph{certainly true}.

Now, we can have the following assignment using Example~\ref{example:beliefcomparison}: I find it \emph{likely} that the dog belongs to Paula and not Quinn (thus, $e(\Be(p\wedge{\sim}q))=\text{‘likely’}$) and \emph{unlikely} that the dog belongs to Quinn, not Paula (i.e., $e(\Be({\sim}p\wedge q))=\text{‘unlikely’}$). Hence, we conclude that I find that the dog belongs to Paula rather than to Quinn i.e., ${\sim}\triangle(\Be(p\wedge{\sim}q)\rightarrow\Be({\sim}p\wedge q))$, \emph{certainly true}.
\end{example}
We are now ready to present the calculus for $\QG$ and prove its completeness.
\begin{definition}[$\HQG$]\label{def:HQG}
The calculus $\HQG$ has the following axioms and rules.
\begin{description}
\item[$\mathsf{nontriv}$:] ${\sim}\triangle(\Be\phi\rightarrow\Be\chi)$ for any $\phi$ and $\chi$ s.t.\ $\CPL\vdash\phi$ and $\CPL\vdash{\sim}\chi$.
\item[$\mathsf{reg}$:] $\Be\phi\rightarrow\Be\phi'$ with $\CPL\vdash\phi\supset\phi'$\footnote{Note that $\mathbf{K}$ is not valid: $\Be(p\supset\bot)\rightarrow(\Be p\rightarrow\Be\bot)$ is easy to disprove. It is, however, valid \emph{on single-point models}. Moreover, $\triangle(\Be\bot\rightarrow\Be\phi)$ which corresponds to the second condition on $\mu$ is provable in $\HbiG$ from $\mathsf{reg}$.};
\item[$\biG$:] instantiations of $\HbiG$ axioms and rules with $\LQG$ formulas. 
\end{description}
\end{definition}
\begin{theorem}[Completeness of $\HQG$]\label{theorem:HQGcompleteness}
Let $\Xi\cup\{\alpha\}\subseteq\LQG$. Then
\[\Xi\models_\QG\alpha\text{ iff }\Xi\vdash_{\HQG}\alpha\]
\end{theorem}
\begin{proof}
As regards the soundness part, we just need to check that the axioms are valid. It is clear that if $\phi$ is a~propositional tautology, then $w\vDash\phi$ for any $w\in W$, and if $\chi$ is a~contradictory formula, then $w\nvDash\chi$ for any $w\in W$. But then, it is clear that $\mu(\|\phi|)>\mu(\|\chi\|)$, whence ${\sim}\triangle(\Be\phi\rightarrow\Be\chi)$ is valid. Furthermore, if $\phi\supset\phi'$ is classically valid, then $v(\phi)\subseteq v(\phi')$ in any model, whence, $\Be\phi\rightarrow\Be\phi'$ is valid as well.

For completeness, we reason by contraposition. Assume that $\Xi\nvdash_{\HQG}\alpha$. We can now extend $\Xi$ with the set of all formulas of the form ${\sim}(\top\coimplies\xi)$ with $\xi$ being a~modal axiom composed from the subformulas of $\Xi$ and $\alpha$. Such an extension is possible via applications of $\mathrm{nec}$ (Definition~\ref{def:HHB}) to the modal axioms. Denote the resulting set $\Xi^*$. It is clear that $\Xi^*\nvdash_{\HQG}\alpha$ and that, moreover, by the completeness theorem for bi-G\"{o}del logic, there is a~valuation $e$ s.t.\ $e[\Xi^*]>e(\alpha)$.

It remains to construct a~model falsifying the entailment. The $\biG$ valuation is already given. We proceed as follows. First, we set $W=\mathcal{P}(\Var(\Xi^*\cup\{\alpha\}))$. Then, we define $w\in v(p)$ iff $p\in w$ for any $w\in W$ and extend it to $\|\cdot\|$ in a~usual fashion. Finally, for any $\Be\phi\in\Sf[\Xi^*\cup\{\alpha\}]$, we set $\mu(\|\phi\|)=e(\Be\phi)$. For other $X\subseteq W$, we set $\mu(X)=\sup\{\mu(\|\phi\|):\phi\in\Sf[\Psi^*\cup\{\alpha\}],\|\phi\|\subseteq X\}$. It remains to check that $\mu$ thus defined satisfies Definition~\ref{def:QGsemantics}.

To show that $\mu$ is monotone, let $X\subseteq X'$. If there exist $\phi,\phi'\in\LCPL$ s.t.\ $\|\phi\|=X$ and $\|\phi'\|=X'$, it is clear from the construction of $W$ that $\phi\supset\phi'$ is a~classical tautology, whence, $\Be\phi\rightarrow\Be\phi'$ is an axiom and ${\sim}(\top\coimplies(\Be\phi\rightarrow\Be\phi'))\in\Xi^*$, and thus $\mu(\|\phi\|)\leq\mu(\|\phi'\|)$, as required. Otherwise, recall that
\begin{align*}
\mu(X)&=\sup\{\mu(\|\phi\|):\phi\in\Sf[\Psi^*\cup\{\alpha\}],\|\phi\|\subseteq X\}
\\
\mu(X')&=\sup\{\mu(\|\phi'\|):\phi'\in\Sf[\Psi^*\cup\{\alpha\}],\|\phi'\|\subseteq X'\}
\end{align*}
whence, clearly, $\mu(X)\leq\mu(X')$, as required. Finally, since ${\sim}\triangle(\Be\top\rightarrow\Be\bot)$ is an instance of the $\mathsf{nontriv}$ axiom, we have that $\mu(W)>\mu(\varnothing)$. 
\end{proof}
\subsection{Correspondence theory for weak uncertainty measures}
It is clear that $\QG$ is the logic of all uncertainty frames. However, $\QG$ does not validate some statements regarding measures one usually expects to hold in the classical case. For example, from the classical point of view and if we subscribe to the closed-world assumption, we \emph{know for certain} that the event ‘it rained in Paris on 28.03.2021 or it did not’ (formally, $r\vee{\sim}r$) occurred. However, $\Be(r\vee{\sim}r)$ \emph{is not valid} in $\QG$\footnote{This shows that $\QG$ can distinguish between capacities and generic uncertainty measures: $\Be\top$ is valid on a~frame if its uncertainty measure is a~capacity.} since it can be that $\mu(\|r\vee{\sim}r\|)<1$.

Moreover, if belief is represented as a~generic uncertainty measure or capacity, it is still possible for two incompatible (or even complementary) events to have measure $1$ \emph{at the same time}. Moreover, it is also possible that $\mu(Y\cup Y')>\mu(Y)$ even if $\mu(Y')=0$. In other words, it is possible that the agent's certainty in $\phi\vee\phi'$ is strictly higher than that in $\phi$ even if they are completely certain that $\phi'$ is not the case.

In this section, we will show how to axiomatise these and other conditions.
\begin{convention}
We introduce the following naming conventions for several formulas.
\begin{description}
\item[$1\mathsf{compl}$:] $\triangle\Be p\leftrightarrow{\sim}\Be{\sim}p$
\item[$\mathsf{disj}+$:] $({\sim}\Be(p\wedge p')\wedge{\sim\sim}\Be p\wedge\!{\sim\sim}\Be p')\!\rightarrow\!({\sim}\triangle(\Be(p\vee p')\!\rightarrow\!\Be p)\wedge\!{\sim}\triangle(\Be( p\vee p')\rightarrow\Be p'))$
\item[$\mathsf{disj}0$:] ${\sim}\Be p\rightarrow\triangle(\Be p'\leftrightarrow\Be(p\vee p'))$
\item[$\mathsf{cap}$:] $\Be\top\wedge{\sim}\Be\bot$
\end{description}
\end{convention}
\begin{theorem}\label{theorem:QGcorrespondencetheory}
Let $\mathbb{F}=\langle W,\mu\rangle$ be an uncertainty frame. Then the following equivalences hold
\begin{align*}
\mathbb{F}\models1\mathsf{compl}&\text{ iff }\ \mu(X)=1\Leftrightarrow\mu(W\setminus X)=0\tag{I}\\
\mathbb{F}\models\mathsf{disj}+&\text{ iff }\ \mu(Y\cap Y')=0\text{ and }\mu(Y),\mu(Y')\!>\!0\Rightarrow\mu(Y\cup Y')\!>\!\mu(Y),\mu(Y')\tag{II}\\
\mathbb{F}\models\mathsf{disj}0&\text{ iff }\ \mu(Y)=0\Rightarrow\mu(Y\cup Y')=\mu(Y')\tag{III}\\
\mathbb{F}\models\mathsf{cap}&\text{ iff }\ \mu\text{ is a~capacity}\tag{IV}
\end{align*}
\end{theorem}
\begin{proof}
We consider III, the other cases can be tackled in a~similar manner.

Indeed, let $\mu(Y)=0$ but $\mu(Y\cup Y')\neq\mu(Y')$ for some $Y,Y'\subseteq W$. Now let $v(p)=Y$ and $v(p')=Y'$. Then, it is clear that $e(\Be p)=0$ but $e(\Be p')\neq e(\Be(p\vee p'))$. Hence, $e(\mathsf{disj}0)\neq1$, as required.

For the converse, we assume that $e(\mathsf{disj}0)\!\neq\!1$. But then, $e(\mathsf{disj}0)\!=\!0$ since $\mathsf{disj}0$ is composed of $\triangle$-formulas and ${\sim}$-formulas. Thus $e({\sim}\Be p)\!=\!1$ but $e(\triangle(\Be p'\!\leftrightarrow\!\Be(p\vee p')))\!=\!0$ (i.e., $\mu(\|p\|\cup\|p'\|))\!\neq\!\mu(\|p'\|)$), whence $\mu(\|p\|)\!=\!0$ but $\mu(\|p\|)\cup \|p'\|))\neq\mu(\|p'\|))$, as required.
\end{proof}

Now, if we want to formalise qualitative counterparts of belief functions, we need to transform $\mu\mathbf{PM}$ into an axiom. However, there is no two-layered formula that can formalise $X\subsetneq Y$. Thus, we have to introduce a~new \emph{axiom schema}.
\begin{align}
{\sim}\triangle(\Be\chi\rightarrow\Be\phi)\rightarrow{\sim}\triangle(\Be(\chi\vee\psi)\rightarrow\Be(\phi\vee\psi))\nonumber\\
\text{with }\CPL\vdash\phi\supset\chi,~\CPL\vdash{\sim}(\chi\wedge\psi),\text{ and }\CPL\nvdash\chi\supset\phi\tag{$\QBel$}\label{equ:QB}
\end{align}
\begin{theorem}\label{theorem:QBcorrespondence}
Let $\mathbb{F}=\langle W,\mu\rangle$ be an uncertainty frame. Then $\mu$ satisfies $\mu\mathbf{PM}$ iff $\mathbb{F}\models\QBel$.
\end{theorem}
\begin{proof}
Let $\mu\mathbf{PM}$ hold, and let further $\phi$, $\chi$, and $\psi$ be as in~\eqref{equ:QB}. Thus, $\|\phi\|\subseteq\|\chi\|$ and $\|\chi\|\cap\|\psi\|=\varnothing$. Note also that both ${\sim}\triangle(\Be\chi\rightarrow\Be\phi)$ and ${\sim}\triangle(\Be(\chi\vee\psi)\rightarrow\Be(\phi\vee\psi))$ can have values only in $\{0,1\}$.

Now, if $e({\sim}\triangle(\Be\chi\rightarrow\Be\phi))=1$, then $\mu(\|\phi\|)<\mu(\|\chi\|)$ and, in fact, $\|\phi\|\subsetneq\|\chi\|$. But then $\mu(\|\phi\|\cup\|\psi\|)<\mu(\|\chi\|\cup\|\psi\|)$, whence $\mu(\|\phi\vee\psi\|)<\mu(\|\chi\vee\psi\|)$, and thus $e({\sim}\triangle(\Be(\chi\vee\psi)\rightarrow\Be(\phi\vee\psi)))=1$, as well.

For the converse, let $\mu\mathbf{PM}$ fail for $\mathbb{F}$, and let, in particular, $\mu(X)<\mu(Y)$, $X\subsetneq Y$, and $Y\cap Z=\varnothing$, but $\mu(X\cup Z)\geq\mu(Y\cup Z)$. We show how to falsify $\QBel$.

We let $\|p\|=X$, $\|p\vee q\|=Y$, and $\|{\sim}q\wedge r\|=Z$. Now, it is easy to see that
\[e({\sim}\triangle(\Be(p\vee q)\rightarrow\Be p)\rightarrow{\sim}\triangle(\Be((p\vee q)\vee({\sim}q\wedge r))\rightarrow\Be(p\vee({\sim}q\wedge r))))=0\]
as required.
\end{proof}
\subsection{Logics of qualitative probabilities}
The main objective of this section is to provide a~two-layered axiomatisation of qualitative probabilities. To do this, we need to transform $\mu\mathbf{KPS}_m$ into axioms. This, however, is not straightforward. This is why, we will take a~detour through the classical logic of qualitative probability $\QP$ introduced in~\cite{Gardenfors1975}. The language of $\QP$ is given by the following grammar.
\[\LQP:~\phi\coloneqq p\in\Prop\mid{\sim}\phi\mid(\phi\wedge\phi)\mid(\phi\lesssim\phi).\]
The semantics uses probabilistic frames and models built upon them.
\begin{definition}[Frame semantics for $\QP$~\cite{Gardenfors1975}]\label{def:Gaerdenforssemantics}
A \emph{G\"{a}rdenfors probabilistic frame} is a~tuple $\mathbb{F}=\langle U,\{\PrG_x\}_{x\in U}\rangle$ with $U\neq\varnothing$ and $\{\PrG_x\}_{x\in U}$ being a~family of probability measures on $2^U$. A~\emph{G\"{a}rdenfors model} is a~tuple $\mathfrak{M}=\langle U,\{\PrG_x\}_{x\in U},v\rangle$ with $\langle U,\{\PrG_x\}_{x\in U}\rangle$ being a~frame and $v:\Prop\rightarrow2^U$ being a~valuation that is extended to a~satisfaction relation $\|\cdot\|$ as follows: 
\begin{itemize}
\item $\|p\| = v(p)$;
\item $\|{\sim}\phi\|=U\setminus\|\phi\|$;
\item $\|\phi\land\psi\|=\|\phi\|\cap\|\psi\|$;
\item $\|\phi\lesssim\psi\|=\{x\in U\mid\PrG_x(\|\phi\|)\leq\PrG_x(\|\psi\|)\}$.
\end{itemize}
For any model $\mathfrak{M}$, we say that $\phi$ is \emph{true in $\mathfrak{M}$} ($\mathfrak{M}\models\phi$) iff $\|\phi\|=U$. Furthermore, $\phi$ is \emph{valid in $\mathbb{F}$} ($\mathbb{F}\models\phi$) iff $\phi$ is true in every model on $\mathbb{F}$.
\end{definition}

Furthermore, we can define an additional notion of satisfaction in a~state.
\begin{definition}[Pointed model semantics]\label{def:Gaerdenforssemanticspointed}
Let $\mathfrak{M}$ be a~G\"{a}rdenfors model and $x\in\mathfrak{M}$. We define $\mathfrak{M},x\vDash\phi$ ($\phi$ is true at $x$) as follows.
\begin{itemize}
\item $\mathfrak{M},x\vDash p$ iff $x\in v(p)$.
\item $\mathfrak{M},x\vDash{\sim}\phi$ iff $\mathfrak{M},x\nvDash\phi$.
\item $\mathfrak{M},x\vDash\phi\wedge\phi'$ iff $\mathfrak{M},x\vDash\phi$ and $\mathfrak{M},x\vDash\phi'$.
\item $\mathfrak{M},x\vDash\phi\lesssim\phi'$ iff $\PrG_x(\|\phi\|)\leq\PrG_x(\|\phi'\|)$.
\end{itemize}
In the remainder of the paper, we will call a~tuple $\langle\mathfrak{M},x\rangle$ a~\emph{pointed model}.
\end{definition}
Other connectives and modalities can be introduced in an expected fashion:
\begin{align*}
\phi\vee\phi'&\coloneqq{\sim}({\sim}\phi\wedge{\sim}\phi')&\phi\supset\phi'&\coloneqq{\sim}\phi\vee\phi'&\phi\equiv\phi'&\coloneqq(\phi\supset\phi')\wedge(\phi'\supset\phi)\\
\phi\approx\phi'&\coloneqq(\phi\lesssim\phi')\wedge(\phi'\lesssim\phi)&\top&\coloneqq p\supset p&\phi<\phi'&\coloneqq(\phi\lesssim\phi')\wedge{\sim}(\phi'\lesssim\phi)
\end{align*}

Let us now recall the axiomatisation of $\QP$. For this, we borrow the $\mathtt{E}$-notation from~\cite{Segerberg1971} and~\cite{Gardenfors1975}. This will help us express the Kraft--Pratt--Seidenberg conditions in a~more concise manner.
\begin{convention}[$\mathtt{E}$-notation]
Consider $\LQP$ formulas $\phi_1$, \ldots, $\phi_n$ and $\chi_1$, \ldots, $\chi_n$. Let further, $\phi^\circ\in\{\phi,{\sim}\phi\}$ and $\chi^\circ\in\{\chi,{\sim}\chi\}$. We introduce a~new operator $\mathtt{E}$ and write
\[\phi_1,\ldots,\phi_n\mathtt{E}\chi_1,\ldots,\chi_n\]
to designate that necessarily the same number of $\phi^\circ_i$'s as of $\chi^\circ_j$'s are actually of the form ${\sim}\phi_i$ and ${\sim}\chi_j$, respectively. For example
\begin{align*}
p_1,p_2\mathtt{E}q_1,q_2&\coloneqq((p_1\wedge p_2\wedge q_1\wedge q_2)\vee({\sim}p_1\wedge p_2\wedge{\sim}q_1\wedge q_2)\\
&\vee({\sim}p_1\wedge p_2\wedge q_1\wedge{\sim}q_2)\vee(p_1\wedge{\sim}p_2\wedge q_1\wedge{\sim}q_2)\\&\vee(p_1\wedge{\sim}p_2\wedge{\sim}q_1\wedge q_2)\vee({\sim}p_1\wedge{\sim}p_2\wedge{\sim}q_1\wedge{\sim}q_2))\approx\top
\end{align*}
More formally, we let $M=\{1,\ldots,m\}$, $K,L\subseteq M$ and set
\[\phi_1,\ldots,\phi_m\mathtt{E}\chi_1,\ldots,\chi_m\coloneqq\left(\bigvee\limits^{m}_{i=0}\bigvee\limits_{\scriptsize{\begin{matrix}|K|=i\\|L|=i\end{matrix}}}\left(\bigwedge\limits_{k\in K}{\sim}\phi_k\wedge\bigwedge\limits_{k'\in M\setminus K}\phi_{k'}\wedge\bigwedge\limits_{l\in L}{\sim}\chi_l\wedge\bigwedge\limits_{l'\in M\setminus L}\chi_{l'}\right)\right)\approx\top\]
\end{convention}
The axiomatisation of $\QP$ which we call $\HQP$ expands the classical propositional rules with new axioms and rules concerning $\lesssim$.
\begin{definition}[$\HQP$]\label{def:HQP}
\begin{description}
\item[]
\item[$(PC)$] All propositional tautologies.
\item[$(A0)$] $((\phi_1\equiv\phi_2)\approx\top)\land((\psi_1\equiv\psi_2)\approx\top))\supset((\phi_1\lesssim\psi_1)\equiv(\phi_2\lesssim\psi_2))$.
\item[$(A1)$] $\bot\lesssim\phi$.
\item[$(A2)$] $(\phi\lesssim\psi)\lor(\psi\lesssim\phi)$.
\item[$(A3)$] $\bot<\top$. 
\item[$(A4)_m$] $(\phi_1,\ldots,\phi_m\mathtt{E}\psi_1,\ldots,\psi_m)\land\bigwedge\limits_{i=1}^{m-1}(\phi_i\lesssim\psi_i)\supset(\psi_m\lesssim\phi_m)$
\end{description}
The rules are modus ponens and necessitation:
\begin{align*}
\mathsf{MP}:\dfrac{\phi\supset\chi\quad\phi}{\chi}&&\mathsf{nec}:\dfrac{\vdash\phi}{\vdash\phi\approx\top}
\end{align*}
\end{definition}

Let us consider the modal axioms. $(A0)$ allows for substitutions of ‘believably equivalent’ formulas. Other axioms correspond to the conditions on measures we cited in the beginning of the section. In particular, $(A1)$ corresponds to $\mathbf{Q1}$; $(A2)$ is the linearity condition; $(A3)$ corresponds to $\mathbf{Q2}$. Finally, the family of axioms $(A4)_m$ corresponds to the $\mathbf{KPS}_m$. As expected~\cite[P.179]{Gardenfors1975}, the expansion of the classical propositional logic with these axioms is complete w.r.t.\ all G\"{a}rdenfors' probabilistic frames.

Note that $\LQP$ \emph{does allow} for the nesting of $\lesssim$ while $\LQG$ prohibits the nesting of $\Be$. However, a~specific fragment of $\LQP$ can be embedded into $\LQG$.
\begin{definition}[Embedding of simple inequality formulas]\label{def:QPCQPGembedding}
We define \emph{simple inequality formulas} ($\SIF$'s) using the following grammar ($\chi$ and $\chi'$ do not contain $\lesssim$):
\[\SIF\ni\phi\coloneqq\chi\lesssim\chi'\mid{\sim}\phi\mid(\phi\wedge\phi)\mid(\phi\vee\phi)\mid(\phi\supset\phi)\]
We define a~translation $^\triangle$ of $\SIF$'s into $\LQG$ as follows.
\begin{align*}
(\chi\lesssim\chi')^\triangle&=\triangle(\Be\chi\rightarrow\Be\chi')\\
({\sim}\phi)^\triangle&={\sim}\phi^\triangle\\
(\phi\circ\phi')^\triangle&=\phi^\triangle\circ\phi'^\triangle\tag{$\circ\in\{\wedge,\vee\}$}\\
(\phi\supset\phi')^\triangle&=\phi^\triangle\rightarrow\phi'^\triangle
\end{align*}
\end{definition}

\begin{remark}\label{rem:nonrepresentableproperties}
It is instructive to observe that not all statements about comparing beliefs can be represented as $\SIF$'s and their translations into $\LQG$. Indeed, $\mathsf{cap}$ and $\mathsf{disj}+$ \emph{are not} translations of $\SIF$'s. In fact, ${\sim\sim}\Be p$ stipulates that the agent's belief in $p$ \emph{is positive}. In $\QP$, it can only be expressed as $p>\bot$. However, as we have already mentioned, $\QP$ cannot distinguish between normalised and non-normalised measures. Thus, one could demand that $\PrG_x$'s be not probability measures but any uncertainty measures satisfying Kraft--Pratt--Seidenberg conditions. This means that $p>\bot$ \emph{is stronger than} ${\sim\sim}\Be p$ for the latter is \emph{compatible with} $\triangle(\Be p\leftrightarrow\Be\bot)$.
\end{remark}

In what follows, we will say that a~$\QG$ model $\mathfrak{M}=\langle W,v,\mu,e\rangle$ is a~\emph{$\QPG$ model} if $\mu$ satisfies $\mu\mathbf{KPS}_m$. In other words, in a~$\QPG$ model, the order on $2^W$ induced by $\mu$ is a~qualitative counterpart of a~probability measure.

We can now establish that the translation in Definition~\ref{def:QPCQPGembedding} is indeed faithful. We do this by showing how to transform a~given pointed G\"{a}rdenfors model $\langle\mathfrak{M},x\rangle$ (recall Definition~\ref{def:Gaerdenforssemanticspointed}) into a~$\QPG$ model that satisfies exactly the translations of $\SIF$'s that $\langle\mathfrak{M},x\rangle$ satisfies. And conversely, how to provide a~G\"{a}rdenfors model using a~given $\QPG$ model preserving all satisfied $\SIF$'s.
\begin{definition}[$\mathsf{G}$-counterparts]\label{def:Gcounterparts}
Let $\mathfrak{M}\!=\!\langle U,\{\PrG_x\}_{x\in U},v\rangle$ be a~G\"{a}r\-den\-fors model. A~\emph{$\mathsf{G}$-counter\-part} of a~\emph{pointed model} $\langle\mathfrak{M},x\rangle$ is the $\QPG$ model $\mathfrak{M}_\QPG=\langle U,v,\PrG_x,e\rangle$.
\end{definition}
\begin{remark}
Note that a~$\triangle$-less translation of $\chi\lesssim\chi'$ as $\Be\chi\rightarrow\Be\chi'$ does not preserve truth. Indeed, let $\PrG_x(\|p\|)=0.7$, $\PrG_x(\|q\|)=0.6$, $\PrG_x(\|r\|)=0.5$, $\PrG_x(\|s\|)=0.4$. Then $(p\lesssim q)\supset(r\lesssim s)$ is \emph{true} at $x$ but $e((\Be p\rightarrow\Be q)\rightarrow(\Be r\rightarrow\Be s))=0.4$.
\end{remark}
\begin{lemma}\label{lemma:GaerdenforstoGoedel}
Let $\langle\mathfrak{M},x\rangle$ be a~\emph{pointed G\"{a}rdenfors model} and $\mathfrak{M}_\QPG$ its $\mathsf{G}$-counterpart, then $\mathfrak{M},x\vDash\phi$ iff $e(\phi^\triangle)=1$ for any $\phi\in\SIF$.
\end{lemma}
\begin{proof}
First, it is clear that for any classical formula $\chi$, it holds that $\|\chi\|=v(\chi)$ since $\mathfrak{M}$ and $\mathfrak{M}_\mathsf{G}$ have the same valuation.

We proceed by induction on $\phi$. First, let $\phi\coloneqq(\chi\lesssim\chi')$.
\begin{align*}
\mathfrak{M},x\vDash\chi\lesssim\chi'&\text{ iff }\PrG_x(\|\chi\|)\leq\PrG_x(\|\chi'\|)\\
&\text{ iff }\PrG_x(v(\chi))\leq\PrG_x(v(\chi'))\tag{$\|\chi\|=v(\chi)$}\\
&\text{ iff }e(\Be\chi\rightarrow\Be\chi')=1\\
&\text{ iff }e(\triangle(\Be\chi\rightarrow\Be\chi'))=1
\end{align*}
For the inductive step, we consider $\phi\coloneqq\phi_1\wedge\phi_2$ and $\phi\coloneqq{\sim}\phi'$.
\begin{align*}
\mathfrak{M},x\vDash\phi_1\wedge\phi_2&\text{ iff }\mathfrak{M},x\vDash\phi_1\text{ and }\mathfrak{M},x\vDash\phi_2\\
&\text{ iff }e(\phi^\triangle_1)\text{ and }e(\phi^\triangle_2)=1\tag{by IH}\\
&\text{ iff }e((\phi_1\wedge\phi_2)^\triangle)=1
\end{align*}
\begin{align*}
\mathfrak{M},x\vDash{\sim}\phi'&\text{ iff }\mathfrak{M},x\nvDash\phi'\\
&\text{ iff }e(\phi'^\triangle)\neq1\tag{by IH}\\
&\text{ iff }e(\phi'^\triangle)=0\tag{$\phi'^\triangle$ is a~Boolean combination of $\triangle$-formulas}\\
&\text{ iff }e(({\sim}\phi')^\triangle)=1
\end{align*}
\end{proof}
\begin{definition}[$\QP$-counterparts]\label{def:QPcounterpart}
Let $\mathfrak{M}=\langle W,v,\mu,e\rangle$ be a~$\QPG$ model. Its \emph{$\QP$-counterpart} is any $\QP$ pointed model $\langle\mathfrak{M},w\rangle$ with $w\in W$ s.t.\ $\mathfrak{M}_\mathsf{G}=\langle W,v,\{\pi_{\mu_x}\}_{x\in W}\rangle$ and $\pi_{\mu_x}$ is a~probability measure s.t.\ $\mu(X)\leq\mu(Y)$ iff $\pi_{\mu_x}(X)\leq\pi_{\mu_x}(Y)$ for all $X,Y\subseteq W$.
\end{definition}
\begin{proposition}\label{prop:QPcounterpartsexistence}
For any $\QPG$ model, there exists its $\QP$ counterpart.
\end{proposition}
\begin{proof}
Note that $\mu$ conforms to Kraft--Pratt--Seidenberg conditions~\cite{KraftPrattSeidenberg1959,Scott1964}. Thus, there is a~pro\-ba\-bi\-li\-ty measure on the same set that preserves all orders from $\mu$.
\end{proof}

Note that we do not demand $\QP$-counterparts to be unique as we are able to prove the next statement regardless.
\begin{lemma}\label{lemma:GoedeltoGaerdenfors}
Let $\mathfrak{M}=\langle W,v,\mu,e\rangle$ be a~$\QPG$ model and $\langle\mathfrak{M}_\mathsf{G},w\rangle$ its any counterpart. Then, $e(\phi^\triangle)=1$ iff $\mathfrak{M}_\mathsf{G},w\vDash\phi$ for any $\phi\in\SIF$.
\end{lemma}
\begin{proof}
Let $e(\phi^\triangle(\mathsf{s}^\triangle_1,\ldots,\mathsf{s}^\triangle_n))=1$ with $\mathsf{s}_i=\chi_i\lesssim\chi'_i$ and $\mathsf{s}^\triangle_i=\triangle(\Be\chi\rightarrow\Be\chi')$. Since the measure on the $\QP$ counterpart preserves all order relations from $\mathfrak{M}$, it is clear that $\mathfrak{M},w\vDash\mathsf{s}_i$ iff $e(\mathsf{s}_i)=1$ for all $i\leq n$. But then we have that $e(\phi^\triangle)=1$ iff $\mathfrak{M}_\mathsf{G},w\vDash\phi$ since $e(\mathsf{s}^\triangle_i)\in\{0,1\}$ and G\"{o}del connectives behave classically on values $0$ and $1$.
\end{proof}
\begin{theorem}\label{theorem:QPCQPGembedding}
Let $\phi\in\SIF$. Then $\phi$ is $\QP$ valid iff $\phi^\triangle$ is $\QPG$ valid.
\end{theorem}
\begin{proof}
Immediately from Lemmas~\ref{lemma:GaerdenforstoGoedel} and~\ref{lemma:GoedeltoGaerdenfors}.
\end{proof}

Now, observe that if we instantiate $\phi_i$'s and $\psi_i$'s in $(A4)_m$ with propositional formulas, these formulas are going to be $\SIF$'s. This means that to obtain the axiomatisation of the logic complete w.r.t.\ $\QPG$ frames, we only need to translate $(A4)_m$ into $\LQG$.
\begin{convention}[$\mathtt{E}$-notation for $\QPG$]
Consider $\LCPL$ formulas $\phi_1$, \ldots, $\phi_n$ and $\chi_1$, \ldots, $\chi_n$. Let further, $\phi^\circ\in\{\phi,{\sim}\phi\}$ and $\chi^\circ\in\{\chi,{\sim}\chi\}$. We introduce operator $\EG$ and write
\[\phi_1,\ldots,\phi_n\EG\chi_1,\ldots,\chi_n\]
to designate that necessarily the same number of $\phi^\circ_i$'s as of $\chi^\circ_j$'s are actually of the form ${\sim}\phi_i$ and ${\sim}\chi_j$, respectively.

More formally, we let $M=\{1,\ldots,m\}$, $K,L\subseteq M$, and set
\[\phi_1,\ldots,\phi_m\EG\chi_1,\ldots,\chi_m\coloneqq\triangle\!\left(\!\Be\!\left(\bigvee\limits^{m}_{i=0}\bigvee\limits_{\scriptsize{\begin{matrix}|K|\!=\!i\\|L|\!=\!i\end{matrix}}}\left(\bigwedge\limits_{k\in K}\!{\sim}\phi_k\wedge\!\!\!\!\bigwedge\limits_{k'\!\in\!M\setminus K}\!\!\!\phi_{k'}\wedge\!\bigwedge\limits_{l\in L}{\sim}\chi_l\wedge\!\!\bigwedge\limits_{l'\in M\setminus L}\!\!\!\chi_{l'}\right)\!\!\right)\!\!\leftrightarrow\!\Be\top\!\right)\]

$\EG$ has the same intended meaning as $\mathtt{E}$. Namely, that the measure of $$\bigvee\limits^{m}_{i=0}\bigvee\limits_{\scriptsize{\begin{matrix}|K|=i\\|L|=i\end{matrix}}}\left(\bigwedge\limits_{k\in K}\!{\sim}\phi_k\wedge\!\!\!\bigwedge\limits_{k'\in M\setminus K}\!\!\!\phi_{k'}\wedge\!\bigwedge\limits_{l\in L}{\sim}\chi_l\wedge\!\!\bigwedge\limits_{l'\in M\setminus L}\!\!\!\chi_{l'}\right)$$ is equal to the measure of $\top$.

Finally, we define
\begin{description}
\item[$\mathsf{KPS}_m$:] $(\phi_1,\ldots,\phi_m\EG\chi_1,\ldots,\chi_m)\wedge\bigwedge\limits_{i=1}^{m-1}\triangle(\Be\phi_i\rightarrow\Be\chi_i)\rightarrow\triangle(\Be\chi_m\rightarrow\Be\phi_m)$.
\end{description}
In what follows, we use $\HQPG$ to designate the extension of $\HQG$ with $\mathsf{KPS}_m$ axioms for every $m>0$.
\end{convention}

The next statements are straightforward corollaries from Theorems~\ref{theorem:QGcorrespondencetheory} and~\ref{theorem:QPCQPGembedding}.
\begin{theorem}\label{theorem:QPG}
Let $\mathsf{KPS}=\{\mathsf{KPS}_m:m\in\mathbb{N}\}$ and $\mathbb{F}=\langle W,\mu\rangle$. Then $\mathbb{F}\models\mathsf{KPS}$ iff $\mu$ satisfies Kraft--Pratt-Seidenberg conditions.
\end{theorem}
\begin{convention}
\label{conv:axiomaticextension}
Let $\mathcal{H}$ be a~Hilbert-style calculus and $\Phi$ be a~scheme of formulas. We denote with $\mathcal{H}\otimes\Phi$ the calculus obtained by adding $\Phi$ to $\mathcal{H}$ as an \emph{axiom scheme}. We also say that $\mathcal{H}$ is \emph{the logic of a~class $\mathbb{K}$ of frames} iff
\[\Xi\vdash_{\mathcal{H}}\alpha\text{ iff }\Xi\models_{\mathbb{K}}\alpha\]
\end{convention}
\begin{theorem}\label{theorem:QGextensionscompleteness}
\begin{enumerate}
\item[]
\item $\HQG\otimes\mathsf{1compl}$ is the logic of the frames satisfying condition (I) from Theorem~\ref{theorem:QGcorrespondencetheory}.
\item $\HQG\otimes\mathsf{disj+}$ is the logic of the frames satisfying condition (II) from Theorem~\ref{theorem:QGcorrespondencetheory}.
\item $\HQG\otimes\mathsf{disj0}$ is the logic of the frames satisfying condition (III) from Theorem~\ref{theorem:QGcorrespondencetheory}.
\item $\HQG\otimes\mathsf{cap}$ is the logic of the frames whose measure is a~capacity.
\item $\HQG\otimes\QBel$ is the logic of the frames whose measure satisfies $\mu\mathbf{PM}$.
\item $\HQPG$ is the logic of $\QPG$ frames.
\end{enumerate}
\end{theorem}
\begin{proof}
All the proofs can be conducted in a~similar manner. This is why, we provide only the most instructive case --- that of $\HQPG$. Soundness follows immediately from Theorem~\ref{theorem:QPG}. For completeness, we reason by contraposition.

Assume that $\Xi\nvdash_{\HQPG}\alpha$. We extend $\Xi$ with all formulas of the form ${\sim}(\top\coimplies\xi)$ with $\xi$ being an instance of a~modal axiom composed from $\Sf[\Xi\cup\{\alpha\}]$ and denote the resulting set $\Xi^*$. Since all such formulas are theorems in $\HQPG$, it is clear that $\Xi^*\nvdash_{\HQPG}\alpha$. But then, by the strong completeness of $\biG$, we have that there is a~$\biG$ valuation $e$ s.t.\ $e[\Xi^*]>e(\alpha)$. Furthermore, it is clear that $e({\sim}(\top\coimplies\xi))=1$ (whence, $e(\xi)=1$) for every $\xi$.

We now need to construct the falsifying model. The valuation is already given. Now, we set $W=\mathcal{P}(\Var(\Xi^*\cup\{\alpha\}))$. Then, we define $w\in v(p)$ iff $p\in w$ for any $w\in W$ and extend it to $\|\cdot\|$ in a~usual fashion. Finally, for any $\Be\phi\in\Sf[\Xi^*\cup\{\alpha\}]$, we set $\mu(\phi)=e(\Be\phi)$. It is clear that $\mu$ is defined on \emph{a subalgebra of $W$ over set union, intersection and complement} and that it satisfies $\mathbf{Q1}$--$\mathbf{Q3}$ and $\mathbf{KPS}_m$ from Theorem~\ref{theorem:qualitativecharacterisations}. But then, there is a~(possibly non-normalised) probability measure $\mathtt{p}_\mu$ on this subalgebra that agrees with $\mu$ on the order. Thus, we can extend $\mathtt{p}_\mu$ to the entire $W$, and clearly, the extended measure will satisfy $\mathbf{Q1}$--$\mathbf{Q3}$ and $\mathbf{KPS}_m$ because it is a~probability measure.
\end{proof}

\section{Paraconsistent logics of monotone comparative belief\label{sec:MCB}}
As we have discussed in the previous section, it is not necessary that an agent be able to assign a~number to their certainty in a~given statement. Furthermore, it is possible that the evidence regarding a~given statement is contradictory or incomplete, whence if we want to compare our certainty in different statements, we need to treat evidence in favour and evidence against independently. Consider, e.g., the following situation: all sources give contradictory information regarding $\phi$ but no information regarding $\chi$, both our certainty that $\phi$ is true and our certainty in its falsity should be greater than our certainty in either truth or falsity of $\chi$.

These characteristics of evidence can be illustrated in the context of court proceedings. Indeed, the evidence in court has the features listed above: it is rare that one can reliably measure one's certainty in any given piece thereof, instead, the court tries to establish whose claims are more compelling; the evidence presented by witnesses can be incomplete or inconsistent; in any court proceeding there are two parties, and the non-contradictory evidence can be treated as favouring one of them (or irrelevant to the process).

This is why we analyse these contexts using paraconsistent G\"{o}del logics introduced in~\cite{BilkovaFrittellaKozhemiachenko2021TABLEAUX} on the outer layer (while $\BD$ is the inner-layer logic). Since G\"{o}del logic can be seen as a~logic of comparative truth, its paraconsistent expansion with a~De Morgan negation $\neg$ can be seen as a~logic of comparative truth and falsity. To make the paper self-contained, we recall $\BD$ and $\Gsquare$.

In the present paper, as well as in~\cite{BilkovaFrittellaKozhemiachenkoMajerNazari2023APAL}, we employ the interpretation of paraconsistent uncertainty from~\cite{KleinMajerRad2021}. Other approaches thereto can be found in, e.g.~\cite{Bueno-SolerCarnielli2016,RodriguesBueno-SolerCarnielli2021}. In these papers, the authors introduce an extension of $\BD$ with operators $\circ$ (classicality) and $\bullet$ (non-classicality) which they call $\mathsf{LET_F}$ (the logic of evidence and truth). It is worth mentioning that the proposed axioms of probability are very close to those from~\cite{KleinMajerRad2021}: e.g., both allow measures $\mathtt{p}$ s.t.\ $\mathtt{p}(\phi)+\mathtt{p}(\neg\phi)<1$ (if the information regarding $\phi$ is incomplete) or $\mathtt{p}(\phi)+\mathtt{p}(\neg\phi)>1$ (when the information is contradictory). The main difference between our approach and that of~\cite{Bueno-SolerCarnielli2016,RodriguesBueno-SolerCarnielli2021} is that Bueno-Soler et al.\ do not introduce a~logic that would formalise reasoning with paraconsistent probabilities. Flaminio, Godo, and Marchioni~\cite{FlaminioGodoMarchioni2011} give an overview of various systems of reasoning about uncertain events using various frameworks of fuzzy logics. They focus on uncertainty measures --- in particular capacities (plausibility measures in the terminology of the article) and possibility measures and represent them as modal fuzzy logics using some core fuzzy logic with the $\triangle$ operator as a~background. The article deals with quantitative uncertainty and does not discuss the question of its qualitative counterpart.

Flaminio et al.~\cite{FlaminioGodoUgolini2022} use measures of the inconsistency of modal theories over H\'{a}jek's framework for probabilistic reasoning as a~modal theory over Lukasiewicz fuzzy logic which they extend with rational truth constants. Their framework represents inconsistent uncertain information, but unlike the approach of~\cite{KleinMajerRad2021}, they do not address the problem of the representation of incomplete information. Neither do they deal with the question of qualitative probabilities.

\subsection{Belnap--Dunn logic}
Belnap--Dunn logic ($\BD$) or first-degree entailment was introduced in~\cite{Belnap2019} as a~four-valued paraconsistent logic over the following language, interpreted on the lattice $\mathbf{4}$ (fig.~\ref{fig:BDsquare}).
\[\LBD:~\phi\coloneqq p\in\Prop\mid\neg\phi\mid(\phi\wedge\phi)\mid(\phi\vee\phi)\]
Recall that the values designate four kinds of information a~computer (or a~database) could have regarding a~given statement $\phi$, or what a~computer might be told regarding it:
\begin{itemize}
\item told \textbf{only} that $\phi$ is \textbf{true} --- $t$;
\item told \textbf{both} that $\phi$ is \textbf{true and} that it is \textbf{false} --- $b$;
\item told \textbf{neither} that $\phi$ is true nor that it is false --- $n$;
\item told \textbf{only} that $\phi$ is \textbf{false} --- $f$.
\end{itemize}
In this paper, we will present $\BD$ using an alternative semantics of $\BD$ using Kripke-style models. 
\begin{definition}[Belnap--Dunn models]
A \emph{Belnap--Dunn model} is a~tuple $\mathfrak{M}=\langle W,v^+,v^-\rangle$ with $W\neq\varnothing$ and $v^+,v^-:\Prop\rightarrow {\mathcal{P}}(W)$. 
\end{definition}
Here, each state of the model can be thought of as a~source of information. It is clear that regarding each statement, a~source can either say that it is true, false, say nothing at all, or provide a~contradictory account. This is formalised in the following definition.
\begin{definition}[Frame semantics for $\BD$]\label{def:BDframesemantics}
Let $\phi,\phi'\in\LBD$. For a~model $\mathfrak{M}=\langle W,v^+,v^-\rangle$, we define notions of $w\vDash^+\phi$ and $w\vDash^-\phi$ (\emph{positive and negative supports}, respectively) for $w\in W$ as follows.
\begin{align*}
w\vDash^+p&\text{ iff }w\in v^+(p)&w\vDash^-p&\text{ iff }w\in v^-(p)\\
w\vDash^+\neg\phi&\text{ iff }w\vDash^-\phi&w\vDash^-\neg\phi&\text{ iff }w\vDash^+\phi\\
w\vDash^+\phi\wedge\phi'&\text{ iff }w\vDash^+\phi\text{ and }w\vDash^+\phi'&w\vDash^-\phi\wedge\phi'&\text{ iff }w\vDash^-\phi\text{ or }w\vDash^-\phi'\\
w\vDash^+\phi\vee\phi'&\text{ iff }w\vDash^+\phi\text{ or }w\vDash^+\phi'&w\vDash^-\phi\vee\phi'&\text{ iff }w\vDash^-\phi\text{ and }w\vDash^-\phi'
\end{align*}
We denote the \emph{positive and negative interpretations} of a~formula as follows:
\begin{align*}
|\phi|^+\coloneqq\{w\in W\mid w\vDash^+\phi\}&&|\phi|^-\coloneqq\{w\in W\mid w\vDash^-\phi\}.
\end{align*}
$|\phi|^+$ (resp., $|\phi|^-$) can be interpreted as the set of the sources saying that the statement $\phi$ is true (resp., false).

We say that a~sequent $\phi\vdash\chi$ is \emph{valid on $\mathfrak{M}=\langle W,v^+,v^-\rangle$} (denoted, $\mathfrak{M}\models[\phi\vdash\chi]$) iff for any $w\in W$, it holds that:
\begin{itemize}
\item if $w\vDash^+\phi$, then $w\vDash^+\chi$ as well;
\item if $w\vDash^-\chi$, then $w\vDash^-\phi$ as well.
\end{itemize}
A sequent $\phi\vdash\chi$ is \emph{universally valid} iff it is valid on every model. In this case, we will say that $\phi$ \emph{entails} $\chi$.
\end{definition}
\begin{convention}
To facilitate the presentation, we define $\LIT=\Prop\cup\{\neg p:p\in\Prop\}$ and denote
\begin{align*}
\Var(\phi)&=\{p\in\Prop:p\text{ occurs in }\phi\}, 
&\LIT(\phi)&=\{l\in\LIT:l\text{ occurs in }\phi\}.
\end{align*}
We also denote $\Sf(\phi)$ the set of all subformulas of $\phi$.
\end{convention}
One can see that $\BD$ inherits truth and falsity conditions from the classical logic but treats them independently. In particular, $w$ claims that $\phi$ is false ($w\!\vDash^-\!\phi$) whenever it claims that its negation is true ($w\!\vDash^+\!\neg\phi$); likewise, a~conjunction is false when at least one of its members is false, etc. This is why there are no universally true, nor universally false formulas. Indeed, one can check that any $\phi(p_1,\ldots,p_n)$ is \emph{neither true nor false} at $w$ if $w\!\nvDash^+\!p_i$ and $w\!\nvDash^-\!p_i$ for all~$p_i$'s.

Furthermore, $\neg\phi\vee\chi$ cannot be treated as a~shorthand for implication. First, it lacks modus ponens: if both $\phi$ and $\neg\phi\vee\chi$ are true, it does not entail that $\chi$ is true as well (for if $\phi$ is both true and false but $\chi$ is not true, then $\neg\phi\vee\chi$ is still true). Second, the deduction theorem fails: even if $\phi$ entails $\chi$, the sequent $\neg\phi\vdash\chi$ is still not valid. In fact, even if $\phi\wedge\phi'$ entails $\chi$, then it is not always the case that $\phi$ entails $\neg\phi'\vee\chi$, for instance, $p\wedge q\vdash q$ is valid while $p\vdash q\vee\neg q$ is not. The converse fails as well: $p\vdash p\vee q$ is valid but $p\wedge\neg p\vdash q$ is not.
\begin{example}\label{example:BDsemantics}
Let us clarify which information corresponds to which truth value using the following example. Assume that we read an announcement about a~dog being lost by its owner.
\begin{quote}
\emph{A female golden retriever was lost on the 5th of October. The last time I saw her on the 7th of October, she had a~wide blue made of leather. Any finder is kindly requested to call +33625153633.}
\end{quote}
It is clear that it is \emph{true only} that the dog lost was a~golden retriever, and \emph{false only} that it was male. However, it is \emph{both true and false} that the owner lost her on the 5th of October: \emph{the announcement contradicts itself} saying that the owner saw the dog for the last time \emph{two days after} the loss. Furthermore, since there is \emph{no information} regarding the location where the dog was lost (or seen last), a~statement such as ‘the dog was lost near the city theatre’ would be \emph{neither true nor false}.
\end{example}
%
$\BD$ can be completely axiomatised\footnote{Alternative axiomatisations can be found in~\cite{OmoriWansing2017,Prenosil2018PhD,Shramko2021}.} using the following calculus $\Rfde$~\cite{Dunn2000}. 
\begin{align*}
\phi\wedge\chi&\vdash\phi&\phi&\vdash\phi\vee\chi\\
\phi\wedge\chi&\vdash\chi&\chi&\vdash\phi\vee\chi\\
\neg(\phi\wedge\chi)&\dashv\vdash\neg\phi\vee\neg\chi&\neg(\phi\vee\chi)&\dashv\vdash\neg\phi\wedge\neg\chi\\
\neg\neg\phi&\dashv\vdash\phi&\phi\wedge(\chi\vee\psi)&\vdash(\phi\wedge\chi)\vee(\phi\wedge\psi)
\end{align*}
\[\vee_e\dfrac{\phi\vdash\psi\quad\chi\vdash\psi}{\phi\vee\chi\vdash\psi}\quad\quad\quad\wedge_i\dfrac{\phi\vdash\chi\quad\phi\vdash\psi}{\phi\vdash\chi\wedge\psi}\]
\begin{theorem}\label{theorem:BDformulaformulacompleteness}
$\phi\vdash\chi$ is provable in $\Rfde$ iff it is valid.
\end{theorem}
\begin{proof}
From~\cite{Dunn2000}, we know that $\Rfde$ is complete w.r.t. the four-valued semantics. We show that $\phi\vdash\chi$ is valid on $\mathbf{4}$ iff it is valid w.r.t.\ the frame semantics. Indeed, assume, $v:\Prop\rightarrow\mathbf{4}$ invalidates $\phi\vdash\chi$, i.e., $v(\phi)\not\leq_\mathbf{4}v(\chi)$. We let $W=\{w\}$ and construct $v^+,v^-:\Prop\rightarrow\{\varnothing,\{w\}\}$ as follows.
\begin{align}
w\in v^+(p)&\text{ iff }v(p)\in\{t,b\}\nonumber\\
w\in v^-(p)&\text{ iff }v(p)\in\{f,b\}\label{equ:twomodels}
\end{align}
Now, it is easy to show by induction on $\psi$, that
\begin{align*}
w\in v^+(\psi)&\text{ iff }v(\psi)\in\{t,b\}\tag{for any $\psi$}\\
w\in v^-(\psi)&\text{ iff }v(\psi)\in\{f,b\}\tag{for any $\psi$}
\end{align*}
But then it is clear that if $v(\phi)\not\leq_\mathbf{4}v(\chi)$, then $\mathfrak{M}_\mathbf{4}\not\models[\phi\vdash\chi]$.

For the converse, we have two options, either (i) there is $w\in\mathfrak{M}$ s.t.\ $w\vDash^+\phi$ and $w\nvDash^+\chi$, or (ii) there is $w'\in\mathfrak{M}$ s.t.\ $w'\vDash^-\chi$ and $w'\nvDash^-\phi$. Again, we use~\eqref{equ:twomodels} to construct $v$ on $\mathbf{4}$ but now, using $v^+$ and $v^-$. The result now follows.
\end{proof}
\subsection{Paraconsistent expansions of $\biG$}
To treat support of truth and support of falsity independently, we expand the language of $\biG$ with De Morgan negation $\neg$. This makes the interpretation of formulas two-dimensional: the first coordinate stands for the support of truth and the second one for the support of falsity. In a~sense, $\Gsquare$ is a~G\"{o}del logic interpreted on the twist-product algebra $[0,1]^{\Join}$ (recall fig.~\ref{fig:NSProb}). While De Morgan negations of $\wedge$ and $\vee$ are unambiguous, there are several options for treating the De Morgan negation of $\rightarrow$~\cite{Wansing2008}.

We choose two of them: (1) an intuitive option, ‘implication is false when its antecedent is true and consequent is false’ (i.e., the way we disprove classical implication), this is the way it is done in Nelson's logic~\cite{Nelson1949} (we denote this implication $\weakrightarrow$); and (2) via co-implication as it is done by Moisil\footnote{This logic was introduced several times: by Wansing~\cite{Wansing2008} as $\mathsf{I}_4\mathsf{C}_4$ and then by Leitgeb~\cite{Leitgeb2019} as HYPE. Cf.~\cite{OdintsovWansing2021} for a~recent and more detailed discussion.} in~\cite{Moisil1942} (this implication will be denoted $\rightarrow$). The falsity conditions of co-implications ($\coimplies$ and $\weakcoimplies$) are dual to those of implications ($\rightarrow$ and~$\weakrightarrow$).

The second versions of implication and co-implication produce a~completely self-dual logic.
\begin{definition}[Language and semantics of $\Gsquare$]
\label{def:semantics:G2}
We fix a~countable set $\Prop$ of propositional letters and consider the following language:
\[\phi\coloneqq p\in\Prop\mid\neg\phi\mid(\phi\wedge\phi)\mid(\phi\vee\phi)\mid(\phi\rightarrow\phi)\mid(\phi\Yleft\phi)\mid(\phi\weakrightarrow\phi)\mid(\phi\weakcoimplies\phi)\]
Let $e_1,e_2:\mathsf{Prop}\rightarrow [0,1]$. We extend them as follows.
\begin{longtable}{rclrcl}
$e_1(\neg\phi)$&$=$&$e_2(\phi)$&$e_2(\neg\phi)$&$=$&$e_1(\phi)$\\
$e_1(\phi\wedge\phi')$&$=$&$e_1(\phi)\wedge_\mathsf{G}e_1(\phi')$&$e_2(\phi\wedge\phi')$&$=$&$e_2(\phi)\vee_\mathsf{G}e_2(\phi')$\\
$e_1(\phi\vee\phi')$&$=$&$e_1(\phi)\vee_\mathsf{G}e_1(\phi')$&$e_2(\phi\vee\phi')$&$=$&$e_2(\phi)\wedge_\mathsf{G}e_2(\phi')$\\
$e_1(\phi\rightarrow\phi')$&$=$&$e_1(\phi)\!\Grightarrow\!e_1(\phi')$&$e_2(\phi\rightarrow\phi')$&$=$&$e_2(\phi')\Gcoimplies e_2(\phi)$\\
$e_1(\phi\coimplies\phi')$&$=$&$e_1(\phi)\Gcoimplies e_1(\phi')$&$e_2(\phi\coimplies\phi')$&$=$&$e_2(\phi')\!\Grightarrow\!e_2(\phi)$\\
$e_1(\phi\weakrightarrow\phi')$&$=$&$e_1(\phi)\!\Grightarrow\!e_1(\phi')$&$e_2(\phi\weakrightarrow\phi')$&$=$&$e_1(\phi)\wedge_\mathsf{G}e_2(\phi')$\\
$e_1(\phi\weakcoimplies\phi')$&$=$&$e_1(\phi)\Gcoimplies e_1(\phi')$&$e_2(\phi\weakcoimplies\phi')$&$=$&$e_2(\phi)\!\vee_\mathsf{G}\!e_1(\phi')$
\end{longtable}
\end{definition}
We will consider two separate logics: $\Gsquareorder$ and $\GsquareNelson$ and their respective languages $\LGsquareorder$ and $\LGsquareNelson$ which have only one set of (co-)implications indicated in the brackets.
\begin{convention}
In what follows, we will use several shorthands ($p$ is a~fixed fresh variable).
\begin{align*}
\toporder&\coloneqq p\rightarrow p&\topNelson&\coloneqq p\weakrightarrow p\\
\botorder&\coloneqq p\coimplies p&\botNelson&\coloneqq p\weakcoimplies p\\
{\sim}_\mathbf{0}p&\coloneqq p\rightarrow\botorder&{\sim}_\mathsf{N}p&\coloneqq p\weakrightarrow\botNelson
\end{align*}
\end{convention}

We define entailments in $\Gsquare$ as follows.
\begin{definition}[$\Gsquare$ entailments]\label{def:G2entailments}
Let $\Gamma\cup\{\phi\}\subseteq\LGsquareorder$ and $\Delta\cup\{\chi\}\subseteq\LGsquareNelson$. We define:
\begin{align*}
\Gamma\models_{\Gsquareorder}\phi&\text{ iff }\forall e_1,e_2:\inf\{e_1(\psi):\psi\in\Gamma\}\leq e_1(\phi)\text{ and }\sup\{e_2(\psi):\psi\in\Gamma\}\geq e_2(\phi)\\
\Delta\models_{\GsquareNelson}\chi&\text{ iff }\forall e_1:\inf\{e_1(\psi):\psi\in\Delta\}\leq e_1(\chi)
\end{align*}
\end{definition}
As we see, $\GsquareNelson$ entailment considers only the first coordinate on $[0,1]^{\Join}$ of the valuation while $\Gsquareorder$ takes into consideration both of them\footnote{There is a~technical reason for such definition as well: notice that if $v(p)=(1,1)$, then $v(p\weakrightarrow p)=(1,1)$ as well. Thus, theorems of $\GsquareNelson$ are not always evaluated at $(1,0)$.}.
\begin{convention}
In what follows, we will write $e$ as a~shorthand of $(e_1,e_2)$ and $e(\phi)$ instead of $(e_1(\phi),e_2(\phi))$ when there is no risk of confusion.
\end{convention}

Note that both $\Gsquare$ entailments are paraconsistent in the following sense.
\begin{proposition}\label{prop:Gsquareparaconsistent}~
\begin{enumerate}
\item Let $\phi\in\LGsquareorder$ be \emph{non-valid}. Then there is $\chi\in\LGsquareorder$ s.t.\ $\phi,\neg\phi\not\models_{\Gsquareorder}\chi$.
\item Let $\phi\in\LGsquareNelson$ and let there be $e_1$ and $e_2$ s.t.\ $e_1(\phi)=1$ and $e_2(\phi)=1$. Then there is $\chi\in\LGsquareNelson$ s.t.\ $\phi,\neg\phi\not\models_{\GsquareNelson}\chi$.
\end{enumerate}
\end{proposition}
\begin{proof}
We prove the first statement. The second can be obtained in the same way. Let $q\notin\Var(\phi)=0$ and set $\chi=q$. Now let $e$ be the valuation s.t.\ $e(\phi)\neq(1,0)$ and $e(q)=(0,1)$. It is clear that $e$ refutes the entailment.
\end{proof}
\begin{remark}\label{rem:constants}
Note that $\mathbf{0}$ and $\mathbf{1}$ s.t.\ $e(\mathbf{1})=(1,0)$ and $e(\mathbf{0})=(0,1)$ are definable.
\begin{align*}
\LGsquareNelson:&&\mathbf{0}&\coloneqq\topNelson\weakcoimplies\topNelson&\LGsquareorder:&&\mathbf{0}&\coloneqq p\coimplies p\\
&&\mathbf{1}&\coloneqq\mathbf{0}\weakrightarrow\mathbf{0}&&&\mathbf{1}&\coloneqq p\rightarrow p
\end{align*}
Note that the definitions of $\mathbf{1}$ and $\mathbf{0}$ in $\LGsquareorder$ coincide with $\toporder$ and $\botorder$ which is not the case for $\LGsquareNelson$. When there is no risk of confusion, we will drop the subscripts and write simply $\bot$ or $\top$.

It is also straightforward to verify that $p\weakrightarrow p$ \emph{is not a~constant}. Thus, the standard definition of $\triangle$ cannot be transferred to $\GsquareNelson$ which prevents its intuitive axiomatisation. On the other hand, $\weakcoimplies$ obeys the expected De Morgan law that is dual to that of $\weakrightarrow$ (cf.~Definition~\ref{def:HGsquare}).
\end{remark}
In the remainder of this subsection, we are going to present Hilbert calculi for $\GsquareNelson$ and $\Gsquareorder$ and prove their completeness. The calculi are straightforward expansions of $\HbiG$ with De Morgan axioms.
\begin{definition}[$\HGsquareorder$ and $\HGsquareNelson$]\label{def:HGsquare}
To obtain $\HGsquareorder$, we add the following axiom schemas and rules (below, $\phi\leftrightarrow\chi$ is a~shorthand for $(\phi\rightarrow\chi)\wedge(\chi\rightarrow\phi)$):
\begin{description}
\item[$\mathsf{neg}$.] $\neg\neg\phi\leftrightarrow\phi$
\item[$\mathsf{DeM}\wedge$.] $\neg(\phi\wedge\chi)\leftrightarrow(\neg\phi\vee\neg\chi)$
\item[$\mathsf{DeM}\vee$.] $\neg(\phi\vee\chi)\leftrightarrow(\neg\phi\wedge\neg\chi)$
\item[$\mathsf{DeM}\!\rightarrow$.] $\neg(\phi\rightarrow\chi)\leftrightarrow(\neg\chi\coimplies\neg\phi)$
\item[$\mathsf{DeM}\coimplies$.] $\neg(\phi\coimplies\chi)\leftrightarrow(\neg\chi\rightarrow\neg\phi)$
\end{description}
For $\HGsquareNelson$, we replace $\rightarrow$ with $\weakrightarrow$ and $\coimplies$ with $\weakcoimplies$. We also add $\mathsf{neg}$ and De Morgan laws for $\wedge$ and $\vee$ (with $\weakleftrightarrow$ instead of $\leftrightarrow$). Finally, the De Morgan laws for (co-)implication are as follows.
\begin{description}
\item[$\mathsf{DeM}\!\weakrightarrow$.] $\neg(\phi\weakrightarrow\chi)\weakleftrightarrow(\phi\wedge\neg\chi)$
\item[$\mathsf{DeM}\!\weakcoimplies$.] $\neg(\phi\weakcoimplies\chi)\weakleftrightarrow(\neg\phi\vee\chi)$
\end{description}
\end{definition}
Note that $\HGsquareorder$ and $\HGsquareNelson$ extend $\mathsf{I}_4\mathsf{C}_4$ and $\mathsf{I}_1\mathsf{C}_1$ (Nelson's logic with co-im\-pli\-ca\-tion) from~\cite{Wansing2008} with prelinearity axioms.

Completeness of both calculi w.r.t.\ their algebraic semantics can be established in the following sense.
\begin{theorem}\label{theorem:HG2algebraiccompleteness}
$\HGsquareorder$ and $\HGsquareNelson$ are strongly complete:
\begin{align*}
\Gamma\vdash_{\HGsquareorder}\phi&\text{ iff }\Gamma\models_{\Gsquareorder}\phi&\Gamma\vdash_{\HGsquareNelson}\phi&\text{ iff }\Gamma\models_{\GsquareNelson}\phi
\end{align*} 
\end{theorem}
\begin{proof}
In the appendix.
\end{proof}
\subsection{Language and semantics of $\MCB$ and $\NMCB$}\label{subsec:MCBNMCBlanguagesemantics}
We will treat each atomic modal formula $\Crtn\phi$ (read ‘the agent is certain in the truth of $\phi$’) as a~piece of evidence. The first coordinate supports one party, and the second the other. Pieces of evidence can be combined in different fashions: we can compare our certainty therein using (co-)implication; choose the more or less certain one with $\vee$ and $\wedge$, etc. G\"{o}delian negation represents the countering of a~given statement --- we show that it is absurd. Finally, $\neg$ is the operator that swaps the support of truth and the support of falsity. But in the context of a~court session, if a~statement is used as an argument for one party, then its negation is actually an argument for the other. Thus, we posit that $\neg\Crtn\phi$ is equivalent to $\Crtn\neg\phi$.

Furthermore, there is a~difference between criminal and civil proceedings (as well as arbitrations). Namely, during a~criminal proceeding, a~defendant is pronounced innocent as long as they were able to present conclusive evidence in their favour or counter the evidence of the prosecution. Furthermore, contradictions are usually interpreted in favour of the defendant. Thus, the two parties in the proceeding are not equal in this respect. On the other hand, both parties in a~civil court (say, two relatives settle an inheritance dispute in court) or an arbitration present evidence in their own favour, after which the court determines whose evidence was more compelling.

This difference can be formalised if we recall two $\Gsquare$ logics and their entailments. $\Gsquareorder$ takes into account \emph{both} coordinates of a~given valuation which makes it closer to the reasoning demonstrated in a~civil process. On the other hand, $\GsquareNelson$ takes into account only the first coordinate. Thus, we can associate each coordinate to a~party of the process (for $\GsquareNelson$, the first coordinate stands for the defence, and the second one for the prosecution).

Now, if $\Gamma\cup\{\phi\}$ is a~set of statements concerning some evidence, then entailment relations can be interpreted as preservation of the degree of certainty from premises to the conclusion. We will call $\Crtn$ the modality of ‘monotone comparative belief’. Here, ‘monotone’ means that $\Crtn$ conforms to the underlying $\BD$ entailment in the sense that if $\phi\vdash\chi$ is valid in $\BD$, $\Crtn\phi$ implies $\Crtn\chi$ in the outer-layer logic. ‘Comparative’ relates to the fact that we use G\"{o}del logic which can be thought of as a~logic of comparative truth.

In this section, we provide two logics of monotone comparative belief based on $\Gsquare$: $\MCB=\langle\BD,\Crtn,\Gsquareorder\rangle$ (stands for ‘monotone comparative belief’) and the Nelson-like $\MCB$ designated $\NMCB=\langle\BD,\Crtn,\GsquareNelson\rangle$.
\begin{definition}[Languages of $\MCB$ and $\NMCB$]\label{def:LMCBandLNMCB}
The languages of $\MCB$ and $\NMCB$ ($\LMCB$ and $\LNMCB$, respectively) are defined via the following grammars.
\[\LMCB:~\alpha\coloneqq\Crtn\phi\mid\neg\alpha\mid\alpha\circ\alpha~(\circ\in\{\wedge,\vee,\rightarrow,\coimplies\},\phi\in\LBD)\]
\[\LNMCB:~\alpha\coloneqq\Crtn\phi\mid\neg\alpha\mid\alpha\circ\alpha~(\circ\in\{\wedge,\vee,\weakrightarrow,\weakcoimplies\}, \phi\in\LBD)\]
\end{definition}

Recall once again that $\pi:2^W\rightarrow[0,1]$ is an uncertainty measure on $W$ iff $\pi(X)\leq\pi(Y)$ for all $X,Y\subseteq W$ s.t.\ $X\subseteq Y$ and $\pi(W)>\pi(\varnothing)$.
\begin{definition}[Semantics of $\MCB$ and $\NMCB$]\label{def:MCBandNMCBsemantics}
An $\MCB$ ($\NMCB$) model is a~tuple $\mathscr{M}\!=\!\langle W,v^+\!,v^-\!,\pi,e_1,e_2\rangle$ with $\langle W,v^+,v^-\rangle$ being a~$\BD$ model (cf.~Definition~\ref{def:BDframesemantics}), $\pi:2^W\rightarrow[0,1]$ being an uncertainty measure.

Semantic conditions of atomic $\LMCB$ and $\LNMCB$ formulas are as follows.
\begin{align*}
e_1(\Crtn\phi)&=\pi\left(\{w:w\vDash^+\phi\}\right)=\pi(|\phi|^+)\\
e_2(\Crtn\phi)&=\pi\left(\{w:w\vDash^-\phi\}\right)=\pi(|\phi|^-)
\end{align*}
Values of complex formulas are computed according to Definition~\ref{def:semantics:G2}.

For a~frame $\mathbb{F}=\langle W,\pi\rangle$ on an $\MCB$ ($\NMCB$) model $\mathscr{M}$, we say that $\alpha\in\LMCB$ ($\beta\in\LNMCB$) is valid on $\mathbb{F}$ ($\mathbb{F}\models\alpha$ and $\mathbb{F}\models\beta$, respectively) iff $e(\alpha)\!=\!(1,0)$ ($e_1(\beta)\!=\!1$) for every $e_1$ and $e_2$ on~$\mathbb{F}$. Finally, for $\Psi\cup\{\alpha\}\subseteq\LMCB$ and $\Omega\cup\{\beta\}\subseteq\LNMCB$, we define the same entailment relations as in Definition~\ref{def:G2entailments}.
\end{definition}
\begin{convention}
In what follows, we will use $\strongarrow$ and $\biarrow$ --- congruential versions of $\weakrightarrow$ and $\leftrightarrowtriangle$ --- defined as follows:
\begin{align*}
\alpha\strongarrow\beta&\coloneqq(\alpha\weakrightarrow\beta)\wedge(\neg\beta\weakrightarrow\neg\alpha)\\
\alpha\biarrow\beta&\coloneqq(\alpha\strongarrow\beta)\wedge(\beta\strongarrow\alpha)
\end{align*}
One can check that
\begin{align*}
e_1(\alpha\strongarrow\beta)=1\text{ iff }e_1(\alpha)\leq e_1(\beta)\text{ and }e_2(\alpha)\geq e_2(\beta)\\
e_1(\alpha\biarrow\beta)=1\text{ iff }e_1(\alpha)=e_1(\beta)\text{ and }e_2(\alpha)=e_2(\beta)
\end{align*}
\end{convention}

Note that there is a~very important difference between $\MCB$ and $\NMCB$ on one hand and $\QG$ on the other hand. Namely, in $\QG$ an agent can compare their beliefs in any two given statements. This, however, is not the case in $\MCB$ and $\NMCB$.

To see this, we define two operators for $\alpha\in\LMCB$ and $\beta\in\LNMCB$.
\begin{align*}
\triangleorder\alpha&\coloneqq{\sim}_\mathbf{0}(\toporder\coimplies\alpha)\wedge\neg{\sim_\mathbf{0}\sim_\mathbf{0}}(\toporder\coimplies\alpha)&\triangleNelson\beta&\coloneqq{\sim}_\mathsf{N}(\topNelson\weakcoimplies\beta)
\end{align*}
One can check that
\begin{align*}
e(\triangleorder\alpha)&=
\begin{cases}
(1,0)&\text{iff }e(\alpha)=(1,0)\\
(0,1)&\text{otherwise}
\end{cases}
&
e_1(\triangleNelson\beta)&=
\begin{cases}
1&\text{iff }e_1(\beta)=1\\
0&\text{otherwise}
\end{cases}
\end{align*}
Thus, in contrast to $\triangle(\alpha\rightarrow\alpha')\vee\triangle(\alpha'\rightarrow\alpha)$ that is $\QG$-valid, neither $\triangleorder(\alpha\rightarrow\alpha')\vee\triangleorder(\alpha'\rightarrow\alpha)$ nor $\triangleNelson(\alpha\strongarrow\alpha')\vee\triangleNelson(\alpha'\strongarrow\alpha)$ are valid in $\MCB$ and $\NMCB$, respectively.

Intuitively, this failure of comparability is justified. First, $\alpha$ and $\alpha'$ can be irrelevant to one another. Indeed, we cannot always answer conclusively what we consider more likely: that it will rain tomorrow or that we will find our lost dog. Second, even if the events are related, we are not necessarily able to compare our confidence in them when the evidence is of different nature.

Recall the situation of Paula and Quinn claiming that the dog is theirs. Assume now that Paula shows a~photo of her with the dog on the leash and Quinn shows the (same or at least very similar) leash. Neither piece of evidence is conclusive and without further investigation, it might not be clear whether one is stronger than the other.

Third, if the events are described \emph{classically} (as done in $\QP$, $\QG$, and $\QPG$), then all contradictory events have measure $0$ (or the least possible positive measure). However, if an agent tries to align their beliefs with what they are told by their sources, this is not necessarily the case. Indeed, if I do not have any information at all regarding $p$, then $\pi(|p|^+)=0$ and $\pi(|p|^-)=0$, whence $e(\Crtn(p\wedge\neg p))=(0,0)$. On the other hand, if I have somewhat reliable sources claiming that $q$ is true and some others (less trusted ones) that it is false, then I can posit that $\pi(|q|^+)=0.5$ and $\pi(|q|^-)=0.3$\footnote{Recall, that both in $\biG$ and $\Gsquare$, the exact numbers assigned to our certainty in a~given event are of little importance. What matters is that (in this case) I have some information that suggests that $q$ is true and some that it is false.}, whence $e(\Crtn(q\wedge\neg q))=(0.3,0.5)$. But then my certainty in $p\wedge\neg p$ is incomparable to that in $q\wedge\neg q$. Finally, if \emph{I know for certain} that $r$ is true (i.e., $\pi(|r|^+)=1$ and $\pi(|r|^-)=0$), then $e(\Crtn(r\wedge\neg r))=(0,1)$. Thus my certainty in $r\wedge\neg r$ \emph{is strictly below} that in both $p\wedge\neg p$ and $q\wedge\neg q$.
\subsection{Axiomatisation}
Let us now introduce Hilbert-style calculi for $\MCB$ and $\NMCB$.
\begin{definition}[$\HMCB$]
The calculus consists of the following axioms and rules ($\phi,\chi\in\LBD$ and $\alpha,\beta\in\LMCB$).
\begin{description}
\item[$\HMCB_{\BD}$] $\Crtn\phi\rightarrow\Crtn\chi$ for any $\phi,\chi\in\LBD$ s.t.~$\phi\vdash\chi$ is $\BD$ valid.
\item[$\HMCB_\neg$] $\Crtn\neg\phi\leftrightarrow\neg\Crtn\phi$.
\item[$\Gsquareorder$] all theorems and rules of $\HGsquareorder$ instantiated with $\MCB$ formulas.
\end{description}
\end{definition}
\begin{definition}[$\HNMCB$]
The calculus consists of the following axioms and rules (below, $\phi,\chi\in\LBD$ and $\alpha,\beta\in\LNMCB$).
\begin{description}
\item[$\HNMCB_{\BD}$ ] $\Crtn\phi\strongarrow\Crtn\chi$ for any $\phi,\chi\in\LBD$ s.t.~$\phi\vdash\chi$ is $\BD$ valid.
\item[$\HNMCB_\neg$] $\Crtn\neg\phi\biarrow\neg\Crtn\phi$.
\item[$\GsquareNelson$] all theorems and rules of $\HGsquareNelson$ instantiated with $\NMCB$ formulas.
\end{description}
\end{definition}
As expected, $\HMCB_{\BD}$ and $\HNMCB_{\BD}$ correspond to the monotonicity of $\pi$ while $\HMCB_\neg$ and $\HNMCB_\neg$ establish the connection between the support of truth and support of falsity of a~given $\phi\in\LBD$.

We finish the section by establishing strong completeness results.
\begin{theorem}[Completeness of $\HMCB$ and $\HNMCB$]\label{theorem:HMCBcompleteness}
For any $\Psi\cup\{\alpha\}\subseteq\LMCB$ and $\Omega\cup\{\beta\}\subseteq\LNMCB$, it holds that
\begin{align*}
\Psi\vdash_{\HMCB}\alpha&\text{ iff }\Psi\models_{\MCB}\alpha&\Omega\vdash_{\HNMCB}\beta&\text{ iff }\Omega\models_{\GsquareNelson}\beta
\end{align*} 
\end{theorem}
\begin{proof}
We show only the case of $\HMCB$ since $\HNMCB$ can be proved similarly.

For the soundness part, it suffices to establish validity of $\HMCB_{\BD}$ and $\HMCB_\neg$. Indeed, if $\phi\vdash\chi$ is $\BD$ valid, then $|\phi|^+\subseteq|\chi|^+$ and $|\chi|^-\subseteq|\phi|^-$ for any $v^+$ and $v^-$. Hence, $\pi(|\phi|^+)\leq\pi(|\chi|^+)$ and $\pi(|\phi|^-)\geq\pi(|\chi|^-)$. Thus, $e(\mathsf{C}\phi\rightarrow\mathsf{C}\chi)=(1,0)$, as required.

Likewise, $e_1(\Crtn\neg\phi)=\pi(|\neg\phi|^+)=\pi(|\phi|^-)$ and $e_2(\Crtn\neg\phi)=\pi(|\neg\phi|^-)=\pi(|\phi|^+)$, while $e_1(\neg\Crtn\phi)=e_2(\Crtn\phi)=\pi(|\phi|^-)$ and $e_2(\neg\Crtn\phi)=e_1(\Crtn\phi)=\pi(|\phi|^+)$.

For the completeness part, we reason by contraposition. An \emph{$\HMCB$ prime theory} is $\Pi\subseteq\LMCB$ s.t.\ $\Pi\vdash_{\HMCB}\gamma$ iff $\gamma\in\Pi$ and for any $\gamma\vee\gamma'\in\Pi$ $\gamma\in\Pi$ or $\gamma'\in\Pi$.

Assume now, that $\alpha$ cannot be inferred from $\Psi$. We construct a~model refuting $\Psi\models_{\MCB}\alpha$ Assume an enumeration of all $\MCB$ formulas. We let $\Psi=\Psi_0$ and define
\begin{align*}
\Psi_{n+1}&=
\begin{cases}
\Psi_n\cup\{\alpha_n\}&\text{ iff }\Psi_n,\alpha_n\nvdash_{\HMCB}\alpha\\
\Psi_n&\text{ otherwise}
\end{cases}
\end{align*}
We now define $\Psi^*=\bigcup\limits_{n<\omega}\Psi_n$. It is clear that $\Psi^*$ is a~maximal prime theory that does not contain~$\alpha$, whence $\Psi^*\nvdash_{\HGsquareorder}\alpha$. But observe that all formulas are actually $\LGsquareorder$ formulas with $\Crtn\phi$'s instead of variables. Thus, by Theorem~\ref{theorem:HG2algebraiccompleteness}, there is a~$\Gsquare$ valuation $e$ s.t.\ $e_1[\Psi^*]>e_1(\alpha)$ or $e_2[\Psi^*]<e_2(\alpha)$. It is also clear from Corollary~\ref{cor:strongtriangleadmissibility} that $\triangleorder\xi\in\Psi^*$ for every $\xi$ being an instance of a~modal axiom. Thus, $e$ evaluates all modal axioms with $(1,0)$.

It remains to define $\pi$ and $v^\pm$. We set $W=\mathcal{P}(\LIT(\Psi^*\cup\{\alpha\}))$. Then for any $w\in W$, we let $w\in v^+(p)$ iff $p\in w$ and $w\in v^-(p)$ iff $\neg p\in w$. And finally, for any $\Crtn\phi\in\Sf[\Psi^*\cup\{\alpha\}]$, we set $\pi(|\phi|)=e_1(\Crtn\phi)$. For other $X\subseteq W$, we set $\pi(X)=\sup\{\pi(|\phi|^+):\phi\in\Sf[\Psi^*\cup\{\alpha\}],|\phi|^+\subseteq X\}$ and $\pi(W)=1$ and $\pi(\varnothing)=0$. It is straightforward to check that $\pi$ thus defined conforms to Definition~\ref{def:MCBandNMCBsemantics}.

The result follows.
\end{proof}
\subsection{Extensions}
The logics of monotone comparative belief provided in the previous subsection were, in a~sense, minimal. It is thus instructive to consider their extensions with additional axioms corresponding to additional conditions imposed on $\pi$.

First, observe that since $\BD$ lacks tautologies and universally false formulas, $\mathsf{cap}$ does not have any analogues in $\MCB$, nor in $\NMCB$. In fact, one can prove Theorem~\ref{theorem:HMCBcompleteness} even without requiring that $\pi(W)=1$ and $\pi(\varnothing)=0$\footnote{Note, however, that in this case, $\pi$ is not going to be a~measure.}. This shows that the truth and falsity of $\MCB$ and $\NMCB$ formulas depend only on the order relations between different uncertainty measures of different events, not on the values of these measures. However, in contrast to $\QP$, it is not problematic: if one describes events using $\BD$, there is no event whose uncertainty measure one could know \emph{a~priori}\footnote{One can, however, \emph{express} in $\MCB$ and $\NMCB$ that the agent is completely certain in a~given statement.} just by its description. This, again, is in line with that $\MCB$ and $\NMCB$ formalise reasoning with uncertainty when the agent tries to build their beliefs using only the information provided by their sources as we discussed in Section~\ref{subsec:MCBNMCBlanguagesemantics}. Furthermore, since $\neg$ is not related to the set-theoretic complement, $\mathsf{KPS}_m$ axioms cannot be meaningfully translated either.

Still, we may assert that two events $\phi$ and $\phi'$ are incompatible if $\pi(|\phi\wedge\phi'|^+)=0$ and $\pi(|\phi\wedge\phi'|^-)=1$. Indeed, this statement corresponds to ${\sim}_\mathbf{0}\Crtn(\phi\wedge\phi')$ having value $(1,0)$ or to the formula $\triangleorder{\sim}_\mathbf{0}\Crtn(\phi\wedge\phi')$. To express incompatibility in $\NMCB$, we define the following connective:
\begin{align*}
\triangle^{!\mathsf{N}}\alpha&\coloneqq\triangleNelson(\mathbf{1}\strongarrow\alpha)
\end{align*}
Now $\triangle^{!\mathsf{N}}$ can be used to express that the agent is completely certain in $\phi$ as follows: $\triangle^{!\mathsf{N}}\Crtn\phi$.
It is clear that
\begin{align*}
e_1(\triangle^{!\mathsf{N}}\alpha)&=
\begin{cases}
1&\text{iff }e(\alpha)=(1,0)\\
0&\text{otherwise}
\end{cases}
\end{align*}
Thus, $\triangle^{!\mathsf{N}}{\sim}_\mathsf{N}\Crtn(\phi\wedge\phi')$ corresponds to the incompatibility of $\phi$ and $\phi'$ in $\NMCB$.

The next statement establishes some correspondence results for $\MCB$ and $\NMCB$.
\begin{convention}
We introduce the following naming conventions.
\begin{description}
\item[$\mathsf{disj}+^\neg$:] $(\triangleorder{\sim}_\mathbf{0}\Crtn(p\wedge p')\wedge{\sim}_\mathbf{0}\triangleorder{\sim}_\mathbf{0}\Crtn p\!\wedge\!{\sim}_\mathbf{0}\triangleorder{\sim}_\mathbf{0}\Crtn p')\!\rightarrow\!({\sim}_\mathbf{0}\triangleorder(\Crtn(p\vee p')\!\rightarrow\!\Crtn p)\!\wedge\!{\sim}_\mathbf{0}\triangleorder(\Crtn(p\vee p')\!\rightarrow\!\Crtn p'))$
\item[$\mathsf{disj}0^\neg$:] $\triangleorder{\sim}_\mathbf{0}\Crtn p\rightarrow\triangleorder(\Crtn p'\leftrightarrow\Crtn(p\vee p'))$
\item[$\mathsf{disj}+^{\mathsf{N}}$:] $(\triangleNelson{\sim}_\mathsf{N}\Crtn(p\wedge p')\wedge{\sim}_\mathsf{N}{\sim}_\mathsf{N}\Crtn p\!\wedge\!{\sim}_\mathsf{N}{\sim}_\mathsf{N}\Crtn p')\!\weakrightarrow\!({\sim}_\mathsf{N}\triangleNelson(\Crtn(p\vee p')\!\rightarrow\!\Crtn p)\!\wedge\!{\sim}_\mathsf{N}\triangleNelson(\Crtn(p\vee p')\!\weakrightarrow\!\Crtn p'))$
\item[$\mathsf{disj}0^\mathsf{N}$:] $\triangleNelson{\sim}_\mathsf{N}\Crtn p\weakrightarrow\triangleNelson(\Crtn p'\weakleftrightarrow\Crtn(p\vee p'))$
\end{description}
\end{convention}
\begin{theorem}\label{theorem:MCBcorrespondencetheory}
Let $\mathbb{F}=\langle W,\pi\rangle$ be a~frame. Then the following statements hold.
\begin{align*}
\mathbb{F}\models\mathsf{disj}+^\neg&\ \text{ iff }\ \pi(X\cap X')=0,\pi(Y\cup Y')=1,\pi(X),\pi(X')\neq0,\pi(Y),\pi(Y')\neq1\\&\Rightarrow\pi(X\cup X')>\pi(X),\pi(X')\text{ or }\pi(Y\cap Y')<\pi(Y),\pi(Y')
\tag{I}\\
\mathbb{F}\models\mathsf{disj}0^\neg&\ \text{ iff }\ \pi(X)=0\text{ and }\pi(Y)=1\Rightarrow\pi(X\cup X')=\pi(X')\text{ and }\pi(Y\cap Y')=\pi(Y')\tag{II}\\
\mathbb{F}\models\mathsf{disj}+^\mathsf{N}&\ \text{ iff }\ Y\cap Y'=\varnothing\text{ and }\mu(Y),\pi(Y')\!>\!0\Rightarrow\pi(Y\cup Y')\!>\!\pi(Y),\pi(Y')\tag{III}\\
\mathbb{F}\models\mathsf{disj}0^\mathsf{N}&\ \text{ iff }\ \pi(Y)=0\Rightarrow\pi(Y\cup Y')=\pi(Y')\tag{IV}
\end{align*}
\end{theorem}
\begin{proof}
Analogously to Theorem~\ref{theorem:QGcorrespondencetheory}.
\end{proof}

The formulas from the previous theorem are paraconsistent analogues of those from Theorem~\ref{theorem:QGcorrespondencetheory}. Note, first of all, that we do not translate $1\mathsf{compl}$. It tells that if an agent is completely certain in some statement, then they should completely disbelieve its negation. In a~paraconsistent setting, however, it might be the case that all sources provide contradictory information about $p$, whence this principle is not justified.

Moreover, there is a~considerable difference in the expressivity of $\MCB$ and $\NMCB$. The former takes into account both support of truth and support of falsity, while the latter only support of truth. It means that the properties of the uncertainty measures that can be axiomatised using $\MCB$ are considerably weaker than those axiomatisable in $\QPG$ or $\NMCB$ because every outer-layer formula corresponds not to one but two subsets of the carrier.
\section{Conclusion\label{sec:conclusion}}
In this paper,
we considered the two-layered logic of qualitative probabilities $\QPG$ and its paraconsistent expansions $\MCB$ and $\NMCB$ that formalise reasoning with uncertain, contradictory, or inconclusive information.
For these logics, we provided Hilbert-style axiomatisations and proved their completeness. We showed that the $\SIF$ fragment of G\"{a}rdenfors' logic of qualitative probabilities $\QP$ can be faithfully embedded into $\QPG$. We also defined generalisations of both $\QPG$ and its paraconsistent expansions and proved some results regarding the correspondence between properties of uncertainty measures and two-layered modal formulas.

Still, several questions remain open. First and foremost, Wong conditions and Kraft--Pratt--Seidenberg conditions we gave in Section~\ref{sec:QPC} show which uncertainty measures can be extended to a~belief function or a~probability measure, respectively, while preserving all orders. In~\cite{KleinMajerRad2021} and~\cite{BilkovaFrittellaKozhemiachenkoMajerNazari2023APAL}, paraconsistent counterparts of probabilities and belief functions were provided. A~natural next step would be to establish their qualitative characterisations and provide a~two-layered axiomatisation with $\Gsquare$ as the outer logic.

In~\cite{BilkovaFrittellaMajerNazari2020,KleinMajerRad2021}, paraconsistent counterparts of belief functions and probability measures were introduced under the titles of ‘non-standard belief functions’ and ‘non-standard probabilities’. Neither $\MCB$ nor $\NMCB$, however, have non-standard belief functions or probabilities as their belief measures. This is why it would be instructive to provide a~similar characterisation of non-standard belief functions and probabilities and develop a~logic of non-standard qualitative probabilities.

Second, we established a~faithful translation of $\SIF$'s into $\LQG$. A~natural question is whether there are any other classes of $\LQP$ formulas that can be faithfully translated into $\LQG$. The most evident candidates would be formulas without nesting of $\lesssim$ and formulas where each variable is in the scope of $\lesssim$.

Finally, $\QP$ has finite model property and is decidable. Since proofs in $\QG$ and its expansions, as well as those in $\MCB$ and $\NMCB$, are just derivations from assumptions in $\biG$ and $\Gsquare$, $\QG$, $\MCB$, and $\NMCB$ and their expansions are decidable as well. Thus, a~straightforward and explicit decision procedure along with the complexity evaluation would be desirable to obtain.
\section*{Acknowledgements}
The research of Marta B\'ilkov\'a was supported by grant 22-01137S of the Czech Science Foundation. The research of Sabine Frittella and Daniil Kozhemiachenko was funded by the grant ANR JCJC 2019, project PRELAP (ANR-19-CE48-0006). This research is part of the MOSAIC project financed by the European Union's Marie Sk\l{}odowska-Curie grant No.~101007627.

We are thankful to two anonymous reviewers for their comments that greatly enhanced the quality of the paper.
\bibliographystyle{plain}
\bibliography{references.bib}
\appendix
\section{Completeness proof of $\HGsquareorder$ and $\HGsquareNelson$}
In this section, we prove the completeness of $\HGsquareorder$ and $\HGsquareNelson$ via the equivalence between their algebraic and frame semantics.
\begin{definition}[$\Gsquare$-frames]\label{def:biGframes}
A $\Gsquare$-frame is a~tuple $\mathbb{F}=\langle W,\preccurlyeq\rangle$ with $W\neq\varnothing$ and $\preccurlyeq$ being a~reflexive linear order on $W$.
\end{definition}
\begin{definition}[Models and semantics]\label{def:GN4kripkesemantics}
A model on the frame $\mathbb{F}$ is a~tuple $\mathfrak{M}=\langle\mathbb{F},v^+,v^-\rangle$ with $v^+,v^-:\mathsf{Prop}\to2^W$ (positive and negative valuations) such that if $s\in v^+(p)$ and $s\preccurlyeq s'$, then $s\in v^+(p)$ (and likewise for $v^-$). Using these maps, the positive and negative support of formulas at state $s\in W$ is defined as follows.
\begin{longtable}{rcl}
$\mathfrak{M},s\vDash^+p$& iff &$w\in v^+(p)$\\
$\mathfrak{M},s\vDash^-p$& iff &$w\in v^-(p)$\\
$\mathfrak{M},s\vDash^+\neg\phi$& iff &$\mathfrak{M},s\vDash^-\phi$\\
$\mathfrak{M},s\vDash^-\neg\phi$& iff &$\mathfrak{M},s\vDash^+\phi$\\
$\mathfrak{M},s\vDash^+\phi_1\wedge\phi_2$& iff &$\mathfrak{M},s\vDash^+\phi_1$ and $\mathfrak{M},s\vDash^+\phi_2$\\
$\mathfrak{M},s\vDash^-\phi_1\wedge\phi_2$& iff &$\mathfrak{M},s\vDash^-\phi_1$ or $\mathfrak{M},s\vDash^-\phi_2$\\
$\mathfrak{M},s\vDash^+\phi_1\vee\phi_2$& iff &$\mathfrak{M},s\vDash^+\phi_1$ or $\mathfrak{M},s\vDash^+\phi_2$\\
$\mathfrak{M},s\vDash^-\phi_1\vee\phi_2$& iff &$\mathfrak{M},s\vDash^-\phi_1$ and $\mathfrak{M},s\vDash^-\phi_2$\\
$\mathfrak{M},s\vDash^+\phi_1\rightarrow\phi_2$& iff &$\forall s'\succcurlyeq s:\mathfrak{M},s'\vDash^+\phi_1\Rightarrow\mathfrak{M},s'\vDash^+\phi_2$\\
$\mathfrak{M},s\vDash^-\phi_1\rightarrow\phi_2$& iff &$\exists s'\preccurlyeq s:\mathfrak{M},s'\nvDash^-\phi_1~\&~\mathfrak{M},s'\vDash^-\phi_2$\\
$\mathfrak{M},s\vDash^+\phi_1\Yleft\phi_2$& iff &$\exists s'\preccurlyeq s:\mathfrak{M},s'\vDash^+\phi_1~\&~\mathfrak{M},s'\nvDash^+\phi_2$\\
$\mathfrak{M},s\vDash^-\phi_1\Yleft\phi_2$& iff &$\forall s'\preccurlyeq s:\mathfrak{M},s'\nvDash^-\phi_1\Rightarrow\mathfrak{M},s'\nvDash^-\phi_2$\\
$\mathfrak{M},s\vDash^+\phi_1\weakrightarrow\phi_2$& iff &$\forall s'\succcurlyeq s:\mathfrak{M},s'\vDash^+\phi_1\Rightarrow\mathfrak{M},s'\vDash^+\phi_2$\\
$\mathfrak{M},s\vDash^-\phi_1\weakrightarrow\phi_2$& iff &$\mathfrak{M},s\vDash^+\phi_1~\&~\mathfrak{M},s\vDash^-\phi_2$\\
$\mathfrak{M},s\vDash^+\phi_1\weakcoimplies\phi_2$& iff &$\exists s'\preccurlyeq s:\mathfrak{M},s'\vDash^+\phi_1~\&~\mathfrak{M},s'\nvDash^+\phi_2$\\
$\mathfrak{M},s\vDash^-\phi_1\weakcoimplies\phi_2$& iff &$\mathfrak{M},s\vDash^-\phi_1$ or $\mathfrak{M},s\vDash^+\phi_2$
\end{longtable}
\end{definition}
\begin{definition}[Entailment]\label{def:KripkeG2entailments}
$\Gamma$ (locally) entails $\phi$ in $\Gsquare$ --- denoted $\Gamma\models\phi$ --- iff for any $\mathfrak{M}$ and $s\in\mathfrak{M}$ it holds that
$$\text{if }\mathfrak{M},s\vDash^+[\Gamma]\text{ then }\mathfrak{M},s\vDash^+\phi$$
\end{definition}
\begin{theorem}[$\HGsquareorder$ completeness]\label{theorem:HG2orderKripkecompleteness}
Let $\Gamma\cup\{\phi\}\subseteq\LGsquareorder$. Then $\Gamma\models\phi$ iff there is a~derivation of $\phi$ from $\Gamma$ in $\HGsquareorder$ s.t.\ $\mathsf{nec}$ is applied only to $\HGsquareorder$ theorems.
\end{theorem}
\begin{proof}
Recall first, that adding $\mathsf{prel}$ axioms to $\HHB$ produces a~calculus for G\"{o}del logic with co-implication which is complete w.r.t.\ linear frames. We use De Morgan laws to transform every formula into its negation normal form. Now let $\phi^*$ be a~negation normal form of $\phi$ where each literal $\neg p$ is substituted with a~fresh variable $p^*$. We know from~\cite[Lemma~12]{Wansing2008} that $\phi$ is falsified on a~$\Gsquareorder$ model iff $\phi^*$ is falsified on an $\HB$ model \emph{over the same frame}. Thus, $\HGsquareorder$ is complete w.r.t.\ the class of linear frames.
\end{proof}

The completeness result for $\HGsquareNelson$ can be proved in the same manner.
\begin{theorem}[$\HGsquareNelson$ completeness]\label{theorem:HG2NelsonKripkecompleteness}
Let $\Gamma\cup\{\phi\}\subseteq\LGsquareNelson$. Then $\Gamma\models\phi$ iff there is a~derivation of $\phi$ from $\Gamma$ in $\HGsquareorder$ s.t.\ $\mathsf{nec}$ is applied only to $\HGsquareNelson$ theorems.
\end{theorem}

It now remains to show that every $\Gsquare$ valuations $e_1$ and $e_2$ on $[0,1]^{\Join}$ can be faithfully transformed into valuations on some linear model and vice versa.
\begin{definition}[Model counterpart of a~$\Gsquare$ valuation]\label{def:algebraicallyinducedmodel} Let $\mathbf{v}$ be a~$\Gsquare$ valuation on $[0,1]^{\Join}$. A~model $\mathfrak{M}_\mathbf{v}=\langle\mathbb{Q},\leq_\mathbf{v},v^+_\mathbf{v},v^-_\mathbf{v}\rangle$ is a~\emph{counterpart} of $\mathbf{v}$ if for its valuations $v^+_\mathbf{v}$ and $v^-_\mathbf{v}$ it holds that:
\begin{align*}
v^+_\mathbf{v}(p)=\mathbb{Q}&\text{ iff }\mathbf{v}_1(p)=1&v^-_\mathbf{v}(p)=\mathbb{Q}&\text{ iff }\mathbf{v}_2(p)=1\\
v^+_\mathbf{v}(p)=\varnothing&\text{ iff }\mathbf{v}_1(p)=0&v^-_\mathbf{v}(p)=\varnothing&\text{ iff }\mathbf{v}_2(p)=0\\
v^+_\mathbf{v}(p)\subseteq v^+_\mathbf{v}(q)&\text{ iff }\mathbf{v}_1(p)\leqslant \mathbf{v}_1(q)&v^-_\mathbf{v}(p)\subseteq v^-_\mathbf{v}(q)&\text{ iff }\mathbf{v}_2(p)\leqslant\mathbf{v}_2(q)\\
v^-_\mathbf{v}(p)\subseteq v^+_\mathbf{v}(q)&\text{ iff }\mathbf{v}_2(p)\leqslant\mathbf{v}_1(q)&v^+_\mathbf{v}(p)\subseteq v^-_\mathbf{v}(q)&\text{ iff }\mathbf{v}_1(p)\leqslant\mathbf{v}_2(q)
\end{align*}
\end{definition}
\begin{lemma}\label{lemma:algebratomodel}
Let $\phi,\phi'\in\LGsquareorder\cup\LGsquareNelson$, $\mathbf{v}$ be a~valuation on $[0,1]^{\Join}$, and $\mathfrak{M}_\mathbf{v}$ be a~counterpart of $\mathbf{v}$. Then it holds that
\begin{align*}
\mathbf{v}_1(\phi)=1&\text{ iff }\mathfrak{M}_\mathbf{v}\vDash^+\phi\\
\mathbf{v}_2(\phi)=1&\text{ iff }\mathfrak{M}_\mathbf{v}\vDash^-\phi\\
\mathbf{v}_1(\phi)\leqslant\mathbf{v}_1(\phi')&\text{ iff } v_\mathbf{v}^+(\phi)\subseteq v_\mathbf{v}^+(\phi')\\
\mathbf{v}_2(\phi)\leqslant\mathbf{v}_2(\phi')&\text{ iff } v_\mathbf{v}^-(\phi)\subseteq v_\mathbf{v}^-(\phi')\\
\mathbf{v}_1(\phi)\leqslant\mathbf{v}_2(\phi')&\text{ iff } v_\mathbf{v}^+(\phi)\subseteq v_\mathbf{v}^-(\phi')\\
\mathbf{v}_2(\phi)\leqslant\mathbf{v}_1(\phi')&\text{ iff } v_\mathbf{v}^-(\phi)\subseteq v_\mathbf{v}^+(\phi')
\end{align*}
\end{lemma}
\begin{proof}
We proceed by induction on $\phi$. Then the basis cases of variables and constants hold by construction.

For the induction steps, we consider only the case of $\phi=\psi\weakrightarrow\psi'$. The other ones are straightforward or can be obtained in a~similar manner. In the following, we let $v^\circ_\mathbf{v}$ to stand for the counterpart of $\mathbf{v}_j$.
\begin{align*}
\mathbf{v}_1(\psi\weakrightarrow\psi')=1&\text{ iff }\mathbf{v}_1(\psi)\leqslant\mathbf{v}_1(\psi')\\
&\text{ iff } v^+_\mathbf{v}(\psi)\subseteq v^+_\mathbf{v}(\psi')\tag{by IH}\\
&\text{ iff }\mathfrak{M}_\mathbf{v}\vDash^+\psi\weakrightarrow\psi'
\end{align*}
\begin{align*}
\mathbf{v}_2(\psi\weakrightarrow\psi')=1&\text{ iff }\mathbf{v}_1(\psi)=1\text{ and }e_2(\psi')=1\\
&\text{ iff }\mathfrak{M}_\mathbf{v}\vDash^+\psi\text{ and }\mathfrak{M}_\mathbf{v}\vDash^-\psi'\tag{by IH}\\
&\text{ iff }\mathfrak{M}_\mathbf{v}\vDash^-\psi\weakrightarrow\psi'
\end{align*}
\begin{align*}
\mathbf{v}_1(\psi\weakrightarrow\psi')\leqslant\mathbf{v}_j(\phi')
&\text{ iff }
\left[\begin{matrix}\mathbf{v}_1(\psi)\leqslant\mathbf{v}_1(\psi')
\\\text{and}\\
\mathbf{v}_j(\phi')\geqslant\mathbf{v}_2(\mathbf{0})\end{matrix}\right]
\text{ or }
\left[\begin{matrix}
\mathbf{v}_1(\psi)>\mathbf{v}_1(\psi')\\\text{and}\\
\mathbf{v}_1(\psi')\leqslant\mathbf{v}_j(\phi')
\end{matrix}\right]
\\
&\text{ iff }\left[\begin{matrix}v^+_\mathbf{v}(\psi)\subseteq v^+_\mathbf{v}(\psi')\\\text{and}\\v^\circ_\mathbf{v}(\phi')\supseteq v^-_\mathbf{v}(\mathbf{0})\end{matrix}\right]\text{ or }\left[\begin{matrix}v^+_\mathbf{v}(\psi)\supsetneq v^+_\mathbf{v}(\psi')\\\text{and}\\v^+_\mathbf{v}(\psi')\subseteq v^\circ_\mathbf{v}(\phi')\end{matrix}\right]\tag{by IH}\\
&\text{ iff } v^+_\mathbf{v}(\psi\weakrightarrow\psi')\subseteq v^\circ_\mathbf{v}(\phi')
\end{align*}
\begin{align*}
\mathbf{v}_2(\psi\weakrightarrow\psi')\leqslant\mathbf{v}_j(\phi')&\text{ iff }\mathbf{v}_1(\psi)\leqslant\mathbf{v}_j(\phi')\text{ or }\mathbf{v}_2(\psi')\leqslant\mathbf{v}_j(\phi')\\
&\text{ iff } v^+_\mathbf{v}(\psi)\subseteq v^\circ_\mathbf{v}(\phi')\text{ or }v^-_\mathbf{v}(\psi')\subseteq v^\circ_\mathbf{v}(\phi')\tag{by IH}\\
&\text{ iff } v^-_\mathbf{v}(\psi\weakrightarrow\psi')\subseteq v^\circ_\mathbf{v}(\phi')
\end{align*}
\end{proof}
\begin{definition}[Algebraic counterparts]\label{def:modelinducedvaluation}
Let $\mathfrak{M}=\langle W,\preccurlyeq,v^+,v^-\rangle$ be a~$\Gsquare(\weakrightarrow)$ model. We say that algebraic valuations $v^{\mathfrak{M}}_1$ and $v^\mathfrak{M}_2$ are \emph{counterparts} of $\mathfrak{M}$ if it holds that:
\begin{align*}
v^\mathfrak{M}_1(p)=1&\text{ iff }v^+(p)=W&v^\mathfrak{M}_2(p)=1&\text{ iff }v^-(p)=W\\
v^\mathfrak{M}_1(p)=0&\text{ iff }v^+(p)=\varnothing&v^\mathfrak{M}_2(p)=0&\text{ iff }v^-(p)=\varnothing\\
v^\mathfrak{M}_1(p)\leqslant v^\mathfrak{M}_1(q)&\text{ iff }v^+(p)\subseteq v^+(q)&v^\mathfrak{M}_2(p)\!\leqslant\!v^\mathfrak{M}_2(q)&\text{ iff }v^-(p)\subseteq v^-(q)\\
v^\mathfrak{M}_1(p)\!\leqslant\!v^\mathfrak{M}_2(q)&\text{ iff }v^+(p)\!\subseteq\!v^-(q)&v^\mathfrak{M}_2(p)\leqslant v^\mathfrak{M}_1(q)&\text{ iff }v^-(p)\!\subseteq\!v^+(q)
\end{align*}
\end{definition}
\begin{lemma}\label{lemma:modeltoalgebra}
Let $\phi,\phi'\in\LGsquareorder\cup\LGsquareNelson$. Then, for any $\Gsquare$-model $\mathfrak{M}=\langle\mathbb{F},v^+,v^-\rangle$ and any valuations $v^\mathfrak{M}_1$ and $v^\mathfrak{M}_2$ induced by $\mathfrak{M}$, it holds that:
\begin{align*}
\mathfrak{M}\vDash^+\phi&\text{ iff } v^{\mathfrak{M}}_1(\phi)=1\\
\mathfrak{M}\vDash^-\phi&\text{ iff } v^{\mathfrak{M}}_2(\phi)=1\\
v^+(\phi)\subseteq v^+(\phi')&\text{ iff } e_1^{\mathfrak{M}}(\phi)\leqslant e_1^{\mathfrak{M}}(\phi')\\
v^-(\phi)\subseteq v^-(\phi')&\text{ iff } e_2^{\mathfrak{M}}(\phi)\leqslant e_2^{\mathfrak{M}}(\phi')\\
v^+(\phi)\subseteq v^-(\phi')&\text{ iff } e_1^{\mathfrak{M}}(\phi)\leqslant e_2^{\mathfrak{M}}(\phi')\\
v^-(\phi)\subseteq v^+(\phi')&\text{ iff } e_2^{\mathfrak{M}}(\phi)\leqslant e_1^{\mathfrak{M}}(\phi')\\
\end{align*}
\end{lemma}
\begin{proof}
Analogously to Lemma~\ref{lemma:algebratomodel}.
\end{proof}

We can now finally prove Theorem~\ref{theorem:HG2algebraiccompleteness}.
\begin{proof}[Proof of Theorem~\ref{theorem:HG2algebraiccompleteness}]
The soundness part can be easily proved by a~routine check of axioms and rules. Now, assume that $\Gamma\not\models_{\Gsquareorder}\phi$. Then, by~\cite[Proposition~5]{BilkovaFrittellaKozhemiachenko2021TABLEAUX}, there is a~$\Gsquare$ valuation s.t.\ $e_1[\Gamma]>e_1(\phi)$. Hence, by Lemma~\ref{lemma:algebratomodel}, there is a~model $\mathfrak{M}$ and $w\in\mathfrak{M}$ s.t.\ $\mathfrak{M},w\vDash^+[\Gamma]$ but $\mathfrak{M},w\nvDash^+\phi$. Thus, by Theorem~\ref{theorem:HG2orderKripkecompleteness}, we obtain that $\phi$ is not $\HGsquareorder$ derivable from $\Gamma$.

The case of $\HGsquareNelson$ can be tackled in a~similar manner.
\end{proof}
\begin{corollary}\label{cor:strongtriangleadmissibility}
The following rules are admissible:
\begin{align*}
\dfrac{\HGsquareorder\vdash\alpha}{\HGsquareorder\vdash\triangleorder\alpha}&&\dfrac{\HGsquareNelson\vdash\beta}{\HGsquareNelson\vdash\triangleNelson\beta}
\end{align*}
\end{corollary}
\end{document}